\documentclass[preprint]{elsarticle}
\usepackage{lineno}
\usepackage{graphicx}
\usepackage{bm}
\usepackage{algpseudocode}
\usepackage{amsmath,amssymb,amsfonts}
\usepackage{amsthm}
\usepackage{xcolor}
\usepackage{float}
\usepackage{longtable,pdflscape}
\usepackage{booktabs}
\usepackage{supertabular}
\usepackage{lscape}
\usepackage{accents}
\usepackage[a4paper,total={6in,8in}]{geometry}
\usepackage{algorithmicx}
\usepackage{hyperref}
\usepackage{algorithm}
\usepackage{physics}
\usepackage{cleveref}

\numberwithin{equation}{section}
\modulolinenumbers[5]

\journal{Knowledge-Based Systems}

\newtheorem{theorem}{Theorem}[section]

\newtheorem{lemma}{Lemma}[section]

\newtheorem{remark}{Remark}[section]
\newtheorem{example}{Example}[section]

\begin{document}

	\title{Quantum Simulation of Stokes Flow via Schr{\"o}dingerisation and Artificial Compressibility}
	
	\author{Shi Jin}
	\author{Jiaqi Tang}
	\author{Qilong Zhai}
	\author{Lei Zhang}
	
	
	
	
	
	\begin{abstract}
		Simulating incompressible Stokes flow is essential for studies in microfluidics and low-Reynolds-number hydrodynamics. However, the computational cost of resolving the associated saddle-point problem grows prohibitively with the dimensionality of the problem. In this work, we present a quantum algorithm based on the Schr{\"o}dingerisation technique for the Stokes equations, incorporating an artificial compressibility regularization. The core of our approach is the design of an explicit quantum circuit that encodes the resulting regularized system. The artificial compressibility formulation provides a unified framework for the system, which is then efficiently mapped to a quantum circuit via the Schr{\"o}dingerisation procedure. A rigorous complexity analysis demonstrates the quantum computational advantage of our algorithms in high-dimensional settings, notably an exponential speedup in problem dimensionality. The validity and scalability of the proposed method are corroborated by numerical simulations performed on Qiskit.
	\end{abstract}
	
	
	\maketitle

	\section{Introduction.}
The time-dependent Stokes problem governs creeping incompressible flow in complex geometries such as porous media \cite{Sto1,ST1,Sto2}. With broad applications spanning petroleum engineering, biomedical transport, heat conduction, and microfluidic systems, this model provides a foundational framework for low-Reynolds-number hydrodynamics. A fundamental form is to find the velocity field $\bm{u}(t,\bm{x})$ and pressure field $p(t,\bm{x})$ satisfying the time-dependent Stokes system subject to periodic boundary conditions:
\begin{align}
    \bm{u}_{t} - a\Delta\bm{u} - \nabla p &= \bm{f}, &&\text{in } \Omega \times (0, T], \label{eqn-10111}\\
    \nabla \cdot \bm{u} &= 0, &&\text{in } \Omega \times (0, T], \label{eqn-10112}\\
    \bm{u}(\cdot, 0) &= \bm{u}_{0}, &&\text{in } \Omega,\label{eqn-10114}
\end{align}where $\Omega$ is a polygonal or polyhedral domain in $\mathbb{R}^{d}$, $\bm{f}$ denotes a momentum source term, $a>0$ is the kinematic viscosity. In the subsequent analysis,  we assume that $\bm{f}$ and $\bm{u}^{0}$ are given and sufficiently smooth.

The numerical solution of the Stokes equations has been extensively studied using classical methods such as finite element methods \cite{ST3,ST4,ST2} and finite volume methods \cite{ST5,ST6}. The fundamental challenge persists across these classical schemes: the saddle-point nature of the Stokes system, which mandates satisfying the incompressibility constraint exactly at the discrete level. This coupling not only complicates the design of stable discretizations but also leads to large, ill-conditioned linear systems that are expensive to solve, especially in high dimensions. One influential strategy is the artificial compressibility method \cite{AC1}. The core idea is to relax the strict incompressibility condition by introducing a pseudo pressure, effectively replacing continuity equation with an artificial compressibility equation. This transformation has led to the development of robust numerical schemes that significantly improve the computational efficiency and accuracy \cite{AC3,AC2}. While traditional approaches are well-established for low-dimensional problems, the extension to high-dimensional settings remains computationally prohibitive, as the required resources scale exponentially with the number of dimensions.

	 To overcome this fundamental dimensionality bottleneck, alternative computational paradigms are urgently needed. Quantum computing represents a paradigm shift for computational mathematics \cite{Feyn1986},  offering potential exponential speedups for problems in linear algebra and differential equations \cite{quan5,quan6}. While Hamiltonian simulation techniques have been highly successful for unitary Schr{\"o}dinger-type dynamics \cite{quan1,quan2,10151,lin2022,10152,quan4,quan3}, a fundamental gap remains in handling physically critical systems with non-unitary, dissipative, or non-Hermitian characteristics, as exemplified by the Stokes equations.

Among unitarization techniques, the Schr{\"o}dingerisation method introduced in \cite{Sch2,Sch1} offers a general framework for converting non-Hermitian linear PDEs and ODEs into Schr{\"o}dinger-type systems within an augmented Hilbert space. By leveraging a warped phase transformation and Fourier analysis, it embeds the original dynamics into a higher dimensional unitary quantum evolution. The approach has been expanded to address a wide array of problems, including problems entailing physical boundary or interface conditions \cite{Inter}, linear dynamical systems with inhomogeneous terms \cite{Inh}, iterative linear algebra solvers \cite{Line}, etc. For specific complex PDE systems, obtaining an explicit quantum circuit representation via Schr{\"o}dingerisation holds significant practical relevance in the current noisy intermediate-scale quantum era, enabling concrete simulations on real-world devices \cite{Cir3,Cir1,Cir2}.

Building upon the Schrödingerisation technique, this work presents a unified quantum algorithmic framework for the Stokes equations, incorporating an artificial compressibility regularization. Inspired by the success of classical regularization approaches, our method provides a single, coherent formulation for the coupled system. The core of our approach is the design of an explicit and efficient quantum circuit that encodes the entire regularized system, leveraging the Schrödingerisation procedure for the mapping. We establish a rigorous complexity analysis, which demonstrates a significant quantum computational advantage in high-dimensional settings, notably an exponential speedup in the problem dimensionality. Finally, the practical accuracy, scalability, and feasibility of the proposed algorithm are demonstrated through numerical simulations performed on Qiskit.
	
	The paper is organized as follows. In Section 2, we review two core classical techniques for incompressible flow—the staggered-grid method for stable velocity-pressure coupling and the artificial compressibility method for saddle-point reformulation—which provide the essential foundation for the quantum algorithms. Section 3 introduces the core quantum algorithmic framework. We first employ the Schr{\"o}dingerisation technique to map the governing equations onto a quantum system and derive its Hamiltonian representation. We then detail the explicit quantum circuit construction via the Trotter-Suzuki decomposition of this Hamiltonian. Section 4 provides a rigorous complexity analysis, establishing a theoretical quantum advantage with provable exponential quantum speedup over classical methods for high dimensional problems. Numerical validation of the algorithm and circuit design is demonstrated through concrete examples in Section 5. Finally, Section 6 concludes the paper and discusses future research directions.
\section{Artificial compressibility systems on staggered grids.}
This section provides a short review of two core classes of classical numerical techniques for solving incompressible flow problems: the artificial compressibility method and the staggered-grid method, which provide essential discrete models and algorithmic concepts that underpin the design of the quantum algorithms presented in the subsequent sections.
\subsection{Artificial compressibility.}
The artificial compressibility method is a cornerstone technique for solving incompressible flow problems \cite{AC1,AC2}. Its core idea is to circumvent the computational difficulties associated with the saddle-point structure by relaxing the strict, instantaneous divergence-free constraint, and replacing it with a dynamic, pseudo-evolving relation between the pressure and the velocity divergence. 

\subsubsection{Artificial compressibility system.}
Specifically, We propose to use the following constitutive relation as the artifical compressibility
\begin{equation}\label{eqn-01041}
\nabla \cdot \bm{u} = \varepsilon p,
\end{equation}
where $\varepsilon >0$ is a small artificial compressibility parameter. To derive the governing system, we take the divergence of the momentum equation \eqref{eqn-10111} and leveraging the commutativity of differential operators yields
$$(\nabla \cdot \bm{u})_t-a \Delta (\nabla \cdot \bm{u})-\Delta p = \nabla \cdot \bm{f}.$$
Substituting \eqref{eqn-01041} into the above equation directly provides an evolution equation for the pressure
\begin{equation}\label{eqn-01042}
p_t = (a+\frac{1}{\varepsilon})\Delta p+\frac{1}{\varepsilon} \nabla \cdot \bm{f}.
\end{equation}
The complete system to be solved is therefore the coupled initial-value problem consisting of the momentum equation \eqref{eqn-10111} and the novel pressure equation \eqref{eqn-01041}.
\subsubsection{Error estimate.}
This section presents a convergence analysis of the artificial compressibility system. We consider the steady Stokes problem first: find $(\bm{u},p) \in [H^1_0(\Omega)]^d \times L^2_0(\Omega)$ satisfying
\begin{align}
     - a\Delta\bm{u} - \nabla p &= \bm{f}, &&\text{in } \Omega ,\nonumber \\
    \nabla \cdot \bm{u} &= 0, &&\text{in } \Omega , \nonumber \\
    \bm{u} &= 0, &&\text{on } \partial \Omega, \nonumber
\end{align}
where $$L^2_0(\Omega):=\left\{q \in L^2(\Omega)\middle| \int_{\Omega}q\, {\rm d}x = 0\right\}.$$
The artifical compressibility system: find $(\bm{u}^{\varepsilon},p^{\varepsilon}) \in [H^1_0(\Omega)]^d \times L^2_0(\Omega)$ such that
\begin{align}
     - a\Delta\bm{u}^{\varepsilon} - \nabla p^{\varepsilon} &= \bm{f}, &&\text{in } \Omega ,\nonumber \\
    \nabla \cdot \bm{u}^{\varepsilon} &= \varepsilon p^{\varepsilon}, &&\text{in } \Omega ,\nonumber \\
    \bm{u}^{\varepsilon} &= 0, &&\text{on } \partial \Omega.\nonumber
\end{align}
The convergence analysis is provided as follows, with the proof detailed in \ref{app1}.
\begin{theorem}
  Let $(\bm{u},p)$ solve the incompressible Stokes problem and let $(\bm{u}^{\varepsilon},p^{\varepsilon}) \in [H^1_0(\Omega)]^d \times L^2_0(\Omega)$ solve the artifical compressibility problem above. Assume the Stokes inf-sup condition with constant $\beta > 0$. Then for $\forall~ 0<\varepsilon<\frac{\beta^2}{a}$, we have
$$\bm{u}^{\varepsilon} \to \bm{u},~ p^{\varepsilon} \to p ,$$
with linear rate $O(\varepsilon)$.
\end{theorem}
\begin{theorem}
Let $(\bm{u}, p)$ solve the incompressible time-dependent Stokes system, and let $(\bm{u}^\varepsilon, p^\varepsilon)$ solve the artificial compressibility system with the same initial velocity $\bm{u}^\varepsilon(0) = \bm{u}(0)$.

If
\[
p \in L^2(0,T; L^2(\Omega)),
\]
then
\[
\|\bm{u}^\varepsilon - \bm{u}\|_{L^\infty(0,T;L^2)}^2 + \|\bm{u}^\varepsilon - \bm{u}\|_{L^2(0,T;H_0^1)}^2 + \frac{1}{\varepsilon} \|\nabla \cdot \bm{u}^\varepsilon\|_{L^2(0,T;L^2)}^2 \leq C \varepsilon \|p\|_{L^2(0,T;L^2)}^2.
\]

If in addition
\[
p \in L^\infty(0,T; L^2), \qquad p_t \in L^2(0,T; L^2),
\]
then
\[
\sup_{0 \leq t \leq T} \|
\nabla (\bm{u}^\varepsilon - \bm{u})(t)\|_{L^2}^2 + \int_0^T \|(\bm{u}^\varepsilon - \bm{u})_t\|_{L^2}^2 \, {\rm d}t + \sup_{0 \leq t \leq T} \frac{1}{\varepsilon} \|\nabla \cdot \bm{u}^\varepsilon(t)\|_{L^2}^2 \leq C \varepsilon,
\]
and
\[
\|p^\varepsilon - p\|_{L^2(0,T;L^2)} \leq C \sqrt{\varepsilon}.
\]
\end{theorem}

\subsection{Staggered grid.}
This subsection presents the finite-volume discretization of the staggered grids. The layout of the staggered grid is illustrated in \autoref{Fig.sta}. Using the two-dimensional case as an example, the horizontal momentum equation is discretized at the nodal point $(i+\tfrac{1}{2},j)$
    \[
    u_{t}-a\left(\frac{u_{i+\frac{3}{2},j} - 2u_{i+\frac{1}{2},j} + u_{i-\frac{1}{2},j}}{\Delta x^{2}} + \frac{u_{i+\frac{1}{2},j+1} - 2u_{i+\frac{1}{2},j} + u_{i+\frac{1}{2},j-1}}{\Delta y^{2}}\right) - \frac{p_{i+1,j} - p_{i,j}}{\Delta x} = f_{x}|_{i+\frac{1}{2},j},
    \]
the vertical momentum equation is discretized at the nodal point $(i,j+\tfrac{1}{2})$
\[
v_t-a\left(\frac{v_{i+1,j+\frac{1}{2}} - 2v_{i,j+\frac{1}{2}} + v_{i-1,j+\frac{1}{2}}}{\Delta x^{2}} + \frac{v_{i,j+\frac{3}{2}} - 2v_{i,j+\frac{1}{2}} + v_{i,j-\frac{1}{2}}}{\Delta y^{2}}\right) - \frac{p_{i,j+1} - p_{i,j}}{\Delta y} = f_{y}|_{i,j+\frac{1}{2}}.
\]
The continuity equation is discretized at pressure point $(i,j)$
\begin{equation}
\begin{aligned}
p_t =& (a+\frac{1}{\varepsilon})\left(\frac{p_{i+1,j} - 2p_{i,j} + p_{i-1,j}}{\Delta x^{2}} + \frac{p_{i,j+1} - 2p_{i,j} + p_{i,j-1}}{\Delta y^{2}}\right)\\
&+\frac{1}{\varepsilon} \left(\frac{f_x|_{i+\frac{1}{2},j} - f_x|_{i-\frac{1}{2},j}}{\Delta x} + \frac{f_y|_{i,j+\frac{1}{2}} - f_y|_{i,j-\frac{1}{2}}}{\Delta y}\right).
\end{aligned}
\end{equation}

	\section{Schr{\"o}dingerisation of the Stokes equations.}
Based on the formulation developed in the preceding section, the system we are required to solve is the Stokes system comprising the momentum equation \eqref{eqn-10111} and the pressure evolution equation \eqref{eqn-01042} derived from the artificial compressibility constraint. By applying the Schr{\"o}dingerisation technique \cite{Sch2,Sch1}, these dissipative equations can be transformed into Schr{\"o}dinger-type systems in one higher dimension. This embedding allows the solution of the original Stokes problem to be encoded within the dynamics of a higher-dimensional, unitary quantum simulation. Consequently, quantum algorithms designed for solving the Schr{\"o}dinger equation can be leveraged, achieving exponential speedup potential over classical methods for high-dimensional instances.

	\subsection{Continuous formulation.}
This section presents the quantum algorithm for solving the Stokes flow, developed at the continuous formulation using the Schr{\"o}dingerisation technique.

Specifically, we consider the following decoupled system of equations for each velocity component $u_i,~ i=1,2,\cdots,d$ and the pressure $p$
\begin{equation}\label{eqn-01062}
\begin{aligned}
\frac{\partial u_i}{\partial t} &=  a \Delta u_i + \frac{\partial p}{\partial x_i} + f_i,  &&\text{in } \Omega \times (0, T],  \\
\frac{\partial p}{\partial t} &=  \left(a + \frac{1}{\varepsilon}\right) \Delta p + \frac{1}{\varepsilon} \sum_{i=1}^{d} \frac{\partial f_i}{\partial x_i},  &&\text{in } \Omega \times (0, T], \nonumber
\end{aligned}
\end{equation}
subject to periodic boundary conditions on $\partial \Omega$ and initial conditions $u_i=u_{0i},~ p_0=\frac{1}{\varepsilon}\sum\limits_{i=1}^{d} \frac{\partial u_{0i}}{\partial x_i}$.

      We can apply the warped phase transformation $$w_i(t,\bm{x},q)=e^{-q}u_i(t,\bm{x}),o(t,\bm{x},q) = e^{-q}p(t,\bm{x}),g_i(t,\bm{x})=e^{-q}f_i(t,\bm{x}),$$ where $q>0$, to the Stokes equation

\begin{equation}\label{eqn-01043}
\begin{aligned}
    \partial_{t}w_i(t,\bm{x},q)&=-\partial_q(a\Delta w_i(t,\bm{x},q)-\partial_{x_i}o(t,\bm{x},q)+g_i(t,\bm{x},q)), \\
\partial_t o_i(t,\bm{x},q) &=-\partial_q \left((a+\frac{1}{\varepsilon})\Delta o_i(t,\bm{x},q)+\sum_{i=1}^{d} \frac{\partial}{\partial x_i}g_i(t,\bm{x},q)\right).
\end{aligned}
\end{equation}
By extending to $q<0$ with initial data $$w_i(0,\bm{x},q)=e^{-|q|}u_i(0,\bm{x}),~ o_i(0,\bm{x},q)=e^{-|q|}\sum_{i=1}^{d} \frac{\partial}{\partial x_i}u_i(0,\bm{x}),$$ and applying $\mathcal{F}_q$, the Fourier transform over $q$, to \eqref{eqn-01043}, one obtains Schr\"{o}dinger equations exactly
\begin{equation}
\begin{aligned}
    \partial_{t}\hat{w}_i(t,\bm{x},\eta)&={\rm i}\eta (a\Delta\hat{w}_i(t,\bm{x},\eta)+\partial_{x_i}\hat{o}(\bm{x},\eta)+\hat{g}_i(\bm{x},\eta)),\\
\partial_{t}\hat{o}_i(t,\bm{x},\eta)&={\rm i}\eta \left((a+\frac{1}{\varepsilon})\Delta\hat{o}_i(t,\bm{x},\eta)+\sum_{i=1}^{d} \frac{\partial}{\partial x_i}\hat{g}_i(\bm{x},\eta)\right),
\end{aligned}
\end{equation}
where 
$$\hat{w}_i(t,\bm{x},\eta):=\mathcal{F}_{q}w_i(t,\bm{x},q),~ \hat{o}(\bm{x},\eta):=\mathcal{F}_{q}o(\bm{x},q),~ \hat{g}_i(\bm{x},\eta):=\mathcal{F}_{q}g_i(\bm{x},q),~ i=1,\cdots,d.$$

\subsection{Quantum representation of finite difference operators.}
Building upon the continuous formulation, we now discretize the Hamiltonian comprising differential operators using finite difference methods. This subsection follows \cite{Cir3} to present the binary representation of the difference operators and their corresponding quantum circuit implementation.

We first consider the one-dimensional domain $\Omega:=[0,L]$ uniformly discretized into $N_{x}=2^{n_{x}}$ intervals with mesh size $h=L/N_{x}$. A discrete function $u$ defined on the grid $x_{i}=ih$ ($i=0,\ldots,N_{x}-1$) can be encoded as the quantum state $|u\rangle:=\sum\limits_{j=0}^{2^{n_{x}}-1}u_{j}|j\rangle$. This one-dimensional representation extends to higher dimensions by constructing the full differential operator as tensor products of the one-dimensional operators along each spatial direction.

The first-order shift operators can be defined as follows:
\begin{align}
S^{-} &:= \sum_{j=1}^{2^{n_{x}}-1} \ket{j-1}\bra{j} = \sum_{j=1}^{n_{x}} I^{\otimes (n_{x}-j)} \otimes \sigma_{01} \otimes \sigma_{10}^{\otimes (j-1)} \triangleq \sum_{j=1}^{n_{x}} s_{j}^{-}, \nonumber\\
S^{+} &:= (S^{-})^{\dagger} = \sum_{j=1}^{2^{n_{x}}-1} \ket{j}\bra{j-1} = \sum_{j=1}^{n_{x}} I^{\otimes (n_{x}-j)} \otimes \sigma_{10} \otimes \sigma_{01}^{\otimes (j-1)} \triangleq \sum_{j=1}^{n_{x}} s_{j}^{+},\label{eqn-10131}
\end{align}
thus, the difference operator with periodic boundary conditions can be expressed as follows,
\begin{equation}
D_P^{+}=\frac{S^--I^{\otimes n_x}+\sigma_{10}^{\otimes n_x}}{h},D_P^{-}=\frac{I^{\otimes n_x}-S^+-\sigma_{01}^{\otimes n_x}}{h},D_{P}^{\Delta} = \frac{S^{-} + S^{+} - 2 I^{\otimes n_{x}}+\sigma_{10}^{\otimes n_x}+\sigma_{01}^{\otimes n_x}}{h^{2}}.
\end{equation}

We will employ the following technique described in \cite{Lem1}, which will be utilized multiple times throughout the paper.

\begin{lemma}\label{lem-10131}
Let $\mathcal{H}\cong \mathbb{C}^{2^{n}}$ be a Hilbert space, and let $\ket{a}$, $\ket{b}\in \mathcal{H}$ be any two orthogonal vectors. Consider the two-dimensional subspace $\mathcal{H}_{ab}=\operatorname{span}\left\{\ket{a},\ket{b}\right\}$ generated by them. For any operator $S_{\varepsilon}$ on this subspace

\[
S_{\varepsilon} = \ket{a}\bra{b} + \epsilon \ket{b}\bra{a}, \quad \epsilon = \pm 1,
\]
there exists a unique unitary transformation $B:\mathcal{H}\to\mathcal{H}$ whose restriction to the subspace $\mathcal{H}_{ab}$ is equivalent to a generalized basis rotation, such that the operator $S_{\varepsilon}$ can be block-diagonalized in the new orthonormal basis $\left\{\ket{a^\prime},\ket{b^\prime}\right\}$ induced by $B$ as
    
    \begin{equation}\label{eqn-10123}
    S_{\varepsilon}=B\left[\left(\frac{1+\epsilon}{2}Z+\frac{1-\epsilon}{2}(-{\rm i}Y)\right)\otimes\ket{1}\bra{1}^{\otimes(n-1)}\right]B^\dagger,
    \end{equation}
    where the new basis is given by linear combinations of the original basis:
    
    \[
    \ket{a^\prime}=B\ket{0}\ket{1}^{\otimes(n-1)}=\frac{\ket{a}+\ket{b}}{\sqrt{2}},\quad\ket{b^\prime}=B\ket{1}^{\otimes n}=\frac{\ket{a}-\ket{b}}{\sqrt{2}}.
    \]

\end{lemma}

\begin{remark}\label{rem-1}
   As reformulated in \eqref{eqn-10123}, it follows that $\|S_{\varepsilon}\| = 1$ immediately.
\end{remark}
\subsection{Discretization.}
This subsection discretizes the algorithm using the finite difference operators introduced above, leading to the matrix representation of the Hamiltonian. 

Consider a $d$-dimensional hypercube $\Omega=[0,L]^d$ by employing a tensor-product discretization. Specifically, each spatial direction is uniformly discretized with the same resolution $N_x = 2^{n_x}$, yielding a total of $N_x^d$ grid points located at $\bm{x}_{j_1,\dots,j_d} = (j_1 h, \dots, j_d h)$ with $j_k = 0,\dots,N_x-1$. A multivariate function $u$ on this grid is encoded as the quantum state

\[
|u\rangle := \sum_{j_1,\dots,j_d=0}^{N_x-1} u_{j_1,\dots,j_d} \; |j_1\rangle \otimes |j_2\rangle \otimes \cdots \otimes |j_d\rangle,
\]
where each $|j_k\rangle$ is the computational basis state of a separate register of $n_x$ qubits representing the $k$-th spatial coordinate. Consequently, the full $d$-dimensional state resides in a Hilbert space of $d \cdot n_x$ qubits. On staggered grids with periodic boundary conditions, the Laplacian operator admits the following discrete representation $$H_{\Delta}=\sum_{\alpha=1}^d(D_{P}^{\Delta})_{\alpha},$$
where $(\bullet)_\alpha:=I^{\otimes(d-\alpha)n_x}\otimes\bullet\otimes I^{\otimes(\alpha-1)n_x}.$

Without loss of generality and to avoid the complexity of multiple liftings, we consider a source term with the exponential temporal decay $\partial_t f=-f$. Define $a_1 = a+\frac{1}{\varepsilon}$, solving \eqref{eqn-01043} is equivalent to solving the following homogeneous equation
\begin{equation*}
\begin{aligned}
\partial_t \left(\begin{array}{c}
w_i(t,\bm{x}_{half},q) \\
o(t,\bm{x},q) \\
g_i(t,\bm{x}_{half},q) \\
\sum_{i=1}^{d} \frac{\partial}{\partial x_i}g_i(t,\bm{x}_{half},q)
\end{array}\right)=& \left(\begin{array}{cccc}
aH_{\Delta} & (D^{+}_P)_i & I^{\otimes n_x} & 0\\
0 & a_1 H_{\Delta} & 0 & I^{\otimes n_x} \\
0 & 0 & -I^{\otimes n_x} & 0 \\
0 & 0 & 0 & -I^{\otimes n_x}
\end{array}\right)\left(\begin{array}{c}
w_i(t,\bm{x}_{half},q) \\
o(t,\bm{x},q) \\
g_i(t,\bm{x}_{half},q) \\
\sum_{i=1}^{d} \frac{\partial}{\partial x_i}g_i(t,\bm{x}_{half},q)
\end{array}\right)\\
=&\left(H_1+{\rm i}H_2\right)\left(\begin{array}{c}
w_i(t,\bm{x}_{half},q) \\
o(t,\bm{x},q) \\
g_i(t,\bm{x}_{half},q) \\
\sum_{i=1}^{d} \frac{\partial}{\partial x_i}g_i(t,\bm{x}_{half},q)
\end{array}\right)\\
=&\left(-\partial_q H_1+{\rm i}H_2\right)\left(\begin{array}{c}
w_i(t,\bm{x}_{half},q) \\
o(t,\bm{x},q) \\
g_i(t,\bm{x}_{half},q) \\
\sum_{i=1}^{d} \frac{\partial}{\partial x_i}g_i(t,\bm{x}_{half},q)
\end{array}\right),~ i=1,\cdots,d,
\end{aligned}
\end{equation*}
where $$H_1=\left(\begin{array}{cccc}
aH_{\Delta} & \frac{1}{2}(D^{+}_P)_i & \frac{1}{2}I^{\otimes n_x} & 0\\
-\frac{1}{2}(D^{-}_P)_i  & a_1 H_{\Delta} & 0 & \frac{1}{2}I^{\otimes n_x} \\
\frac{1}{2}I^{\otimes n_x} & 0 & -I^{\otimes n_x} & 0 \\
0 & \frac{1}{2}I^{\otimes n_x} & 0 & -I^{\otimes n_x}
\end{array}\right),$$
$$H_2=\left(\begin{array}{cccc}
0 & -\frac{{\rm i}}{2}(D^{+}_P)_i  & -\frac{{\rm i}}{2}I^{\otimes n_x} & 0\\
\frac{{\rm i}}{2}(D^{-}_P)_i  & 0 & 0 & -\frac{{\rm i}}{2}I^{\otimes n_x} \\
\frac{{\rm i}}{2}I & 0 & 0 & 0 \\
0 & \frac{{\rm i}}{2}I^{\otimes n_x} & 0 & 0 
\end{array}\right),$$
are Hermitian matrices. The state vector could include auxiliary components to pad the total dimension to a power of two, which facilitates efficient quantum encoding. Then
\begin{equation}\label{eqn-10161}
\partial_{t} \left(\begin{array}{c}
\hat{w}_i(t,\bm{x}_{half},\eta) \\
\hat{o}(t,\bm{x},\eta) \\
\hat{g}_i(t,\bm{x}_{half},\eta) \\
\sum_{i=1}^{d} \frac{\partial}{\partial x_i}\hat{g}_i(t,\bm{x}_{half},\eta)
\end{array}\right)={\rm i}(\eta H_1+H_2)\left(\begin{array}{c}
\hat{w}_i(t,\bm{x}_{half},\eta) \\
\hat{o}(t,\bm{x},\eta) \\
\hat{g}_i(t,\bm{x}_{half},\eta) \\
\sum_{i=1}^{d} \frac{\partial}{\partial x_i}\hat{g}_i(t,\bm{x}_{half},\eta)
\end{array}\right).
\end{equation}

Let $v_i(t,\bm{x},q)$ denotes $\left(w_i(t,\bm{x}_{half},q),o(t,\bm{x},q),g_i(t,\bm{x}_{half},q),\sum\limits_{i=1}^{d} \frac{\partial}{\partial x_i}g_i(t,\bm{x}_{half},q)\right)^{T}$, the variables $v_i$ and $\hat{v}_i$ are discretized as $\bm{v}_i(t):=[v_{j_1\cdots j_d,k}(t)]_{j_1\cdots j_d,k}$ and $\hat{\bm{v}}_i(t):=[\hat{v}_{j_1\cdots j_d,k}(t)]_{j_1\cdots j_d,k}$ with 
the initial conditions 

\begin{equation*}
\begin{aligned}
\bm{v}_i(0) &= \left(\bm{u}_i(0),\bm{p},\bm{f}_i,\nabla \cdot \bm{f}\right)^{T} \otimes \left[e^{-|q_0|},\ldots,e^{-|q_{N_q-1}|}\right]\triangleq \tilde{\bm{u}}_i(0) \otimes \bm{q} \\
&= \sum_{\substack{0 \leq j_1,\dots,j_d \leq N_x-1}} \Bigg( \bigg(u_i(0,\bm{x}_{j_1,\cdots,j_i+\frac{1}{2},\cdots,j_d})\ket{00} + \nabla \cdot u_0(\bm{x}_{j_1,\cdots,j_d})\ket{01} \\
&+ f_i(\bm{x}_{j_1,\cdots,j_i+\frac{1}{2},\cdots,j_d})\ket{10} + \nabla \cdot f(\bm{x}_{j_1,\cdots,j_i+\frac{1}{2},\cdots,j_d})\ket{11}\bigg)\ket{j_1}\cdots \ket{j_d} \Bigg) \\
&\otimes \left[e^{-|q_0|},\ldots,e^{-|q_{N_q-1}|}\right],
\end{aligned}
\end{equation*}
and $\hat{\bm{v}}_i(0):=\mathcal{F}_p \bm{v}_i(0)$, where $-\pi R=q_0<\cdots<q_{N_q}=\pi R$ with mesh size $\Delta q=2\pi R/N_q$ and the Fourier variable $\eta_k=(k-\frac{N_q}{2})/R,~ k=0,1,\ldots,N_q-1$.

\subsection{Quantum circuit.}
We proceed to construct the detailed implementation of the quantum circuit for the Hamiltonian evolution operator $ U_{Stokes}(\tau):=\exp({\rm i}H_{Stokes}\tau),~ H_{Stokes}=\eta H_1+H_2$ derived in the previous section. The Schr{\"o}dingerisation is achieved by applying the quantum Fourier transform prior to the time evolution, followed by the inverse quantum Fourier transform to recover the solution in the original space \cite{Cir3}. We apply the first-order Lie-Trotter-Suzuki decomposition to break down the operator into elementary components, each of which admits an efficient quantum circuit representation. One has
\begin{equation*}
\begin{aligned}
U_{Stokes}(\tau) =& \exp\left({\rm i}\tau (H_1 \otimes D_{\eta}+H_2 \otimes I^{\otimes n_q})\right)\\
\approx&\exp\left({\rm i}\tau H_1 \otimes D_{\eta}\right)\cdot \exp({\rm i}\tau H_2 \otimes I^{\otimes n_q})\\
\approx&\exp({\rm i}\tau H_1^{(1)} \otimes D_{\eta})\cdot \exp({\rm i}\tau H_1^{(2)} \otimes D_{\eta})\cdot \exp({\rm i}\tau H_1^{(3)} \otimes D_{\eta})\\
&\cdot \exp({\rm i}\tau H_1^{(4)}\otimes D_{\eta})\cdot \exp({\rm i}\tau H_2^{(1)} \otimes I^{\otimes n_q})\cdot \exp({\rm i}\tau H_2^{(2)} \otimes I^{\otimes n_q}),
\end{aligned}
\end{equation*}
where
$$H_1^{(1)} = a\ket{00}\bra{00}\otimes H_{\Delta}+a_1\ket{01}\bra{01}\otimes H_{\Delta},H_1^{(2)} = \frac{1}{2}\left(\ket{00}\bra{01}\otimes (D_P^{+})_i-\ket{01}\bra{00}\otimes (D_P^{-})_i\right),$$
$$H_1^{(3)} = \frac{1}{2}(\ket{00}\bra{10}+\ket{10}\bra{00}) \otimes I^{dn_x}+\frac{1}{2}(\ket{01}\bra{11}+\ket{11}\bra{01}) \otimes I^{dn_x},$$
$$H_1^{(4)} = -\ket{1}\bra{1} \otimes I^{\otimes (dn_x+1)},~H_2^{(1)} = -\frac{i}{2}\left(\ket{00}\bra{01}\otimes (D_P^{+})_i +\ket{01}\bra{00}\otimes (D_P^{-})_i\right),$$
$$H_2^{(2)} = -\frac{i}{2}(\ket{00}\bra{10}-\ket{10}\bra{00}) \otimes I^{dn_x}-\frac{i}{2}(\ket{01}\bra{11}-\ket{11}\bra{01}) \otimes I^{dn_x},$$
denote
$$U_{Stokes}(\tau)\approx U_{11}^{(1)}U_{12}^{(1)}U_{13}^{(1)}U_{14}^{(1)}U_{21}^{(1)}U_{22}^{(1)}(\tau),U_{ij}^{(1)}(\tau) = \exp\left({\rm i}\tau H_{i}^{(j)}\otimes D_{\eta}^{2-i}\right).$$

We now present the detailed quantum circuit implementation for each component of the algorithm. Regarding the first item
\begin{equation}
\begin{aligned}
\exp\left({\rm i}\tau H_1^{(1)} \otimes D_{\eta}\right) =& \exp({\rm i}a \tau \ket{00}\bra{00} \otimes H_{\Delta}\otimes D_{\eta})\cdot \exp({\rm i}a_1 \tau \ket{01}\bra{01}H_{\Delta}\otimes D_{\eta})\\
=& \left(\ket{00}\bra{00} \otimes \exp({\rm i}a\tau H_{\Delta} \otimes D_{\eta})+(I-\ket{00}\bra{00}) \otimes I^{\otimes (dn_x+n_q)} \right)\\
&\cdot \left(\ket{01}\bra{01} \otimes \exp({\rm i}a_1\tau H_{\Delta} \otimes D_{\eta})+(I-\ket{01}\bra{01}) \otimes I^{\otimes (dn_x+n_q)} \right)\\
\triangleq& U_{11}(a,\tau)U_{11}(a_1,\tau),\nonumber
\end{aligned}
\end{equation}
the only difference between $U_{11}(a_1,\tau)$ and $U_{11}(a,\tau)$ lies in the control bits and phases. We notice the following properties
\begin{equation*}
\exp\left(\sum_{k}A_{k}\otimes\ket{k}\bra{k}\right)=\sum_{k}\exp(A_{k})\otimes\ket{k}\bra{k},\quad\exp\left(\sum_{\alpha}(A)_{\alpha}\right)=\prod_{\alpha}(\exp(A))_{\alpha},
\end{equation*}
then
\begin{align*}
\exp({\rm i}a\tau H_{\Delta} \otimes D_{\eta}) &= \exp({\rm i}a\tau H_{\Delta} \otimes \operatorname{diag}(\eta_0,\ldots,\eta_{N_p-1}))\\
&= \sum_{k=0}^{N_{q}-1}\exp\left({\rm i}a\tau\left(k-\frac{N_{q}}{2}\right)H_{\Delta}\right)\otimes\ket{k}\bra{k} \\
&= \sum_{k=0}^{N_{q}-1}\left(\exp\left({\rm i}a\tau H_{\Delta}\right)\right)^{k-N_{q}/2}\otimes\ket{k}\bra{k} \\
&= \left(\left(\exp\left({\rm i}a\tau H_{\Delta}\right)\right)^{-N_{q}/2}\otimes I^{\otimes n_q}\right) \sum_{k=0}^{N_{q}-1}\left(\exp\left({\rm i}a\tau H_{\Delta}\right)\right)^{k}\otimes\ket{k}\bra{k},
\end{align*}
the latter part of this expression coincides with the formulation for the quantum Fourier transform presented in \cite{lin2022}. We adopt the identical binary representation to achieve the most efficient quantum gate implementation. Let $k=(k_{n_{p}-1}\cdots k_{0})=\sum_{m=0}^{n_{p}-1}k_{m}2^{m}$, then we have

\begin{align*}
\sum_{k=0}^{N_p-1}\exp\left({\rm i}a\tau H\right)^{k}(\tau)\otimes\ket{k}\bra{k} &= \sum_{k_{n_p-1},\ldots,k_0}\prod_{m=0}^{n_p-1}\exp\left({\rm i}a\tau H\right)^{k_m 2^m}(\tau)\otimes\ket{k_{n_p-1}\cdots k_0}\bra{k_{n_p-1}\cdots k_{n_p-1}} \\
&= \mathop{\prod_{m}}\nolimits'\left(\exp\left({\rm i}a\tau H\right)^{2^m}(\tau)\otimes\ket{1}\bra{1}+I^{\otimes n_x}\otimes\ket{0}\bra{0}\right),
\end{align*}
where the primed product $\prod'$ denotes the regular matrices product for the first register (consisting of $n_{x}$ qubits) and the tensor product for the second register (consisting of $n_{p}$ qubits). Assuming that $\exp({\rm i}\tau H)$ can be implemented with a cost independent of $k$, this approach is exponentially more efficient than direct implementation. Under this representation, the implementation of $\exp({\rm i} a\tau H \otimes D_{\eta})$ simplifies to the realization of $\exp({\rm i} a\tau H)$. 

Substituting \eqref{eqn-10131} to $U_{11}(a,\tau)$, we obtain
\begin{align}
\exp\left({\rm i}a\tau \sum_{\alpha=1}^d\left(D_P^\Delta\right)_{\alpha} \right)=&\exp\left(\frac{{\rm i}a\tau}{h^2}\sum_{\alpha=1}^d\left(S^-+S^+-2I^{\otimes n_x}+\sigma_{10}^{\otimes n_x}+\sigma_{01}^{\otimes n_x}\right)_{\alpha}\right)\nonumber\\
=&\exp\left(\frac{{\rm i}a\tau}{h^2}\sum_{\alpha=1}^d\left(\sum_{j=1}^{n_x}(s_j^-+s_j^+)-2I^{\otimes n_x}+\sigma_{10}^{\otimes n_x}+\sigma_{01}^{\otimes n_x}\right)_{\alpha}\right)\nonumber\\
\approx& \exp\left(\frac{{\rm i}a\tau}{h^2}\sum_{\alpha=1}^d\left(\sum_{j=1}^{n_x}(s_j^-+s_j^+)-2I^{\otimes n_x}\right)_{\alpha}\right) \nonumber\\
&\cdot \exp\left(\frac{{\rm i}a\tau}{h^2}\sum_{\alpha=1}^d\left(\sigma_{10}^{\otimes n_x}+\sigma_{01}^{\otimes n_x}\right)_{\alpha}\right).
\end{align}

Since the matrix $H_2$ contains ${\rm i}$, we consider the general combination over the complex field
\begin{equation}\label{eqn-10133}
e^{{\rm i}\lambda}s_{j}^{-}+e^{-{\rm i}\lambda}s_{j}^{+}= I^{\otimes(n_{x}-j)}\otimes\left(e^{{\rm i}\lambda}\ket{0}\bra{1}^{\otimes(j-1)}+e^{-{\rm i}\lambda}\ket{1}\bra{0}^{\otimes(j-1)}\right).
\end{equation}
Recalling \cref{lem-10131}, we can define the unitary matrix 
\begin{equation*}
B_{1j}(\lambda):=\left(\prod_{m=1}^{j-1}\text{CNOT}_{m}^{j}\right)P_{j}(-\lambda)H_{j},
\end{equation*}
such that
\begin{equation*}
\begin{aligned}
B_{1j}(\lambda)\ket{0}\ket{1}^{\otimes(n-1)} &= \frac{\ket{0}\ket{1}^{\otimes(j-1)}+e^{-i\lambda}\ket{1}\ket{0}^{\otimes(j-1)}}{\sqrt{2}}, \\
B_{1j}(\lambda)\ket{1}\ket{1}^{\otimes (n-1)} &= \frac{\ket{0}\ket{1}^{\otimes(j-1)}-e^{-i\lambda}\ket{1}\ket{0}^{\otimes(j-1)}}{\sqrt{2}},
\end{aligned}
\end{equation*}
where $ H_{j} $ is the Hadamard gate acting on the $ j $-th qubit, $ P_{j}(\lambda) $ is the Phase gate acting on the $ j $-th qubit. The CNOT gate $CNOT_{m}^{j}$ acts on the $m$-th qubit controlled by the $j$-th qubit. Then we get the following quantum representation of \eqref{eqn-10133}

\begin{equation*}
e^{{\rm i}\lambda}s_{j}^{-}+e^{-{\rm i}\lambda}s_{j}^{+}=I^{\otimes(n_{x}-j)}\otimes B_{1j}(\lambda)\left(Z\otimes\ket{1}\bra{1}^{\otimes(j-1)}\right)B_{1j}(\lambda)^{\dagger}.
\end{equation*}

The time evolution operator is formulated as

\begin{align*}
\exp\left({\rm i}\tau(e^{{\rm i}\lambda}s_{j}^{-}+e^{-{\rm i}\lambda}s_{j}^{+})\right)
&=I^{\otimes(n_{x}-j)}\otimes B_{1j}(\lambda)CRZ_{j}^{1,\ldots,j-1}(-2\tau)B_{1j}(\lambda)^{\dagger}\\
&\triangleq I^{\otimes(n_{x}-j)}\otimes C_{1j}(\tau,\lambda),
\end{align*}
where $CRZ_{j}^{1,\ldots,j-1}(-2\tau):=\exp({\rm i}\tau Z_{j})\otimes\ket{1}\bra{1}^{\otimes(j-1)}+I\otimes(I^{\otimes(j-1)}-\ket{1}\bra{1}^{\otimes(j-1)})$ is the multi-controlled RZ gate acting on the $j$-th qubit controlled by $1,\ldots,j-1$-th qubits. Applying the first-order Lie-Trotter-Suzuki decomposition, one has
$$
\exp\left({\rm i}\tau\sum_{j=1}^{n_{x}}(e^{i\lambda}s_{j}^{-}+e^{-i\lambda}s_{j}^{+})\right) \approx \prod_{j=1}^{n_{x}}I^{\otimes(n_{x}-j)}\otimes C_{1j}(\tau,\lambda) \triangleq W_{1}(\tau,\lambda),
$$
we denote this simply as $W_1(\tau)$ when $\lambda=0$. The expression $\exp{-2{\rm i}\theta I^{\otimes n_x}}$ is realized in the quantum circuit as a global phase factor. 
Similarly, we have
\begin{equation*}
e^{{\rm i}\lambda}\sigma_{01}^{\otimes n_x}+e^{-{\rm i}\lambda}\sigma_{10}^{\otimes n_x}= B_2(\lambda)(Z\otimes \ket{1}\bra{1}^{\otimes (n_x-1)})B_2(\lambda)^{\dagger},
\end{equation*}
where $B_2(\lambda)=X_{n_x}\left(\prod_{m=1}^{n_x-1}\text{CNOT}_{m}^{n_x}\right)X_{n_x}P_{n_x}(-\lambda)H_{n_x}$, then
$$
\exp\left({\rm i}\tau(e^{{\rm i}\lambda}\sigma_{01}^{\otimes n_x}+e^{-{\rm i}\lambda}\sigma_{10}^{\otimes n_x})\right)=B_{2}(\lambda)CRZ_{n_x}^{1,\ldots,n_x-1}(-2\tau)B_{1j}(\lambda)^{\dagger} \triangleq W_{2}(\tau,\lambda).
$$

As for $H_1^{(3)}$ and $H_2^{(2)}$, we consider the following expression from \cref{lem-10131}
\begin{equation}\label{eqn-10137}
\begin{aligned}
&e^{{\rm i}\lambda}\ket{00}\bra{10}\otimes I^{\otimes dn_x}+e^{-{\rm i}\lambda}\ket{10}\bra{00}\otimes I^{\otimes dn_x}\\
=&X_{n_x}P_{n_x+1}H_{n_x+1}(Z \otimes \ket{1}\bra{1})H_{n_x+1}P_{n_x+1}X_{n_x}\otimes I^{\otimes dn_x},
\end{aligned}
\end{equation}
then we have 
\begin{equation}\label{eqn-01071}
\exp({\rm i}\tau e^{{\rm i}\lambda}\ket{00}\bra{10}\otimes I^{\otimes dn_x}+e^{-{\rm i}\lambda}\ket{10}\bra{00}\otimes I^{\otimes dn_x})=C_3(\tau,\lambda)\otimes I^{\otimes 2n_x}\triangleq W_{3}(\tau,\lambda),
\end{equation}
where $C_3(\tau,\lambda)=X_{n_x}P_{n_x+1}H_{n_x+1}CRZ_{n_x+1}^{n_x}(-2\tau)H_{n_x+1}P_{n_x+1}X_{n_x}$. Similarly, we denote
\begin{equation*}
\exp({\rm i}\tau e^{{\rm i}\lambda}\ket{01}\bra{11}\otimes I^{\otimes dn_x}+e^{-{\rm i}\lambda}\ket{11}\bra{01}\otimes I^{\otimes dn_x})\triangleq W_{7}(\tau,\lambda).
\end{equation*}

Now consider $H_1^{(2)}$ and $H_2^{(1)}$. By substituting \eqref{eqn-10131}, we have
\begin{equation*}
\begin{aligned}
&\exp\left(\frac{\tau}{2}\left(\ket{00}\bra{01}\otimes (D_P^{+})_i-\ket{01}\bra{00}\otimes (D_P^{-})_i\right)\right)\\
\approx &\exp\left(\frac{\tau}{2}\left(\ket{00}\bra{01}\otimes (S^{-})_i-\ket{01}\bra{00}\otimes (S^{+})_i\right)\right) \cdot\exp\left(-\frac{\tau}{2}\left(\ket{00}\bra{01}+\ket{01}\bra{00}\right)\otimes I^{\otimes dn_x}\right)\\
&\cdot \exp\left(\frac{\tau}{2}\left(\ket{00}\bra{01}\otimes (\sigma_{10}^{n_x})_i-\ket{01}\bra{00}\otimes (\sigma_{01}^{n_x})_i\right)\right),
\end{aligned}
\end{equation*}
following the same encoding scheme as in \eqref{eqn-01071}, we can obtain the quantum representation of the following expression, which we denote by $W_4(\tau,\lambda)$.
$$\exp({\rm i}\tau e^{{\rm i}\lambda}\ket{00}\bra{01}\otimes I^{\otimes dn_x}+e^{-{\rm i}\lambda}\ket{01}\bra{00}\otimes I^{\otimes dn_x}).$$
Then we consider the following expression from \cref{lem-10131}
\begin{equation*}
\begin{aligned}
&e^{{\rm i}\lambda}\ket{0}\bra{1}\otimes S^{-} \otimes I+e^{-{\rm i}\lambda}\ket{1}\bra{0}\otimes S^{+}\otimes I\\
=&e^{{\rm i}\lambda}\ket{0}\bra{1}\otimes \sum_{j=1}^{n_x}s_j^{-} \otimes I+e^{-{\rm i}\lambda}\ket{1}\bra{0}\otimes \sum_{j=1}^{n_x}s_j^{+}\otimes I\\
=& \sum_{j=1}^{n_x}B_{3j}(\lambda)\left(Z\otimes \ket{1}\bra{1}^{\otimes n_x}\right)B_{3j}(\lambda)^{\dagger} \otimes I,
\end{aligned}
\end{equation*}
where 
$$B_{3j}(\lambda) =\left(\prod_{m=1}^{j-1}\text{CNOT}_{m}^{n_x+1}\right)X_{n_x+1}\text{CNOT}_j^{n_x+1}X_{n_x+1}P_{n_x+1}(-\lambda)H_{n_x+1},$$
then we have
$$
\exp\left({\rm i}\tau (e^{{\rm i}\lambda}\ket{0}\bra{1}\otimes S^{-} \otimes I+e^{-{\rm i}\lambda}\ket{1}\bra{0}\otimes S^{+}\otimes I)\right) = \prod_{j=1}^{n_x}C_{5j}(\tau,\lambda )\otimes I \triangleq W_{5}(\tau,\lambda),
$$
where $C_{5j}(\tau,\lambda)=B_{3j}(\lambda)CRZ_{j}^{1,\ldots,j-1}(-2\tau)B_{3j}(\lambda)^{\dagger}$.
Similarly, we have
\begin{equation*}
\exp({\rm i} \tau e^{{\rm i}\lambda}\sigma_{01}\otimes \sigma_{10}^{\otimes n_x}+e^{-{\rm i}\lambda}\sigma_{10}\otimes \sigma_{01}^{\otimes n_x})= B_4(\lambda)CRZ^{1,\cdots,n_x-1}_{n_x}(-2\tau)B_4(\lambda)^{\dagger} \triangleq W_{6}(\tau,\lambda),
\end{equation*}
where $B_4(\lambda)=\left(\displaystyle\prod_{m=1}^{n_x}\text{CNOT}_{m}^{n_x+1}\right)P_{n_x+1}(-\lambda)H_{n_x+1}.$

As for $H_1^{(4)}$, we notice that $\exp({\rm i} \tau H_1^{(4)})$ is precisely the phase gate $P(-\lambda)$ applied to the first qubit. We present the following comprehensive complexity analysis of the overall quantum circuit, followed by a schematic diagram of the complete circuit.
\begin{lemma}
   The approximated time evolution operator can be implemented using at most $\mathcal{O}(dN_qn_x^2)$ CNOT gates for $n_x \geq 3$.
\end{lemma}
\begin{proof}
The gate complexity is dominated by the implementation of the multi-controlled rotation gates arising from the finite-difference operators. The key observation is that the majority of these gates act on non-overlapping sets of qubits, allowing for a parallelized implementation. The total CNOT count is obtained by summing over all such gates.

  The implementation of $U_{11}(a,\tau)$ requires at most $2^{n_q-1} \displaystyle\prod_{\alpha=1}^{d} (c^2W_1(-a\tau)c^2W_{2}(-a\tau))_{\alpha}$ gates and $(2^{n_q}-1)$ controlled $\displaystyle\prod_{\alpha=1}^{d}(c^2W_1(a\tau)c^2W_{2}(a\tau))_{\alpha}$ gates, here $c^2$ denotes the presence of two control qubits. Furthermore, $W_1$ consists of a phase gate and $C_{1j},j=1,\cdots,n_x. C_{1j}$ can be decomposed into a multi-controlled RZ gate, 2 Hadmard gates, 2 phase gates and $2(j-1)$ CNOT gates. $W_2$ consists of a multi-controlled RZ gate, 2 Hadmard gates, 2 phase gates, 4 X gates and $2(n_x-1)$ CNOT gates. The total number of CNOT gates in the circuit implementation of $U_{11}(a,\tau)$ is determined by
$$G_{U_{11}(a,\tau)} = d2^{n_q-1}(G_{W_1}+G_{W_2})+d(2^{n_q}-1)(G_{cW_1}+G_{cW_2}).$$
According to the decomposition techniques presented in \cite{Lem1,Lem2}, it is known that a multi-controled RZ or RX gate with $(j-1)$ control qubits can be decomposed into single qubit gates and at most $(16j-40)$ CNOT gates. Then
\begin{align*}
G_{W_1} = &\sum_{j=3}^{n_x+2}(16j-40)+\sum_{j=2}^{n_x}2(j-1)*24\\
= &32n_x^2-24n_x,
\end{align*}
and
\begin{align*}
  G_{cW_1}=&48n_x+\sum_{j=2}^{n_x}2(j-1)*40+\sum_{j=3}^{n_x+3}(16j-40)\\
=&48n_x^2+24n_x+8.
\end{align*}
Then $U_{11}(a,\tau)$ and $U_{11}(a_1,\tau)$ can be implemented using single qubit gates and at most $\mathcal{O}(dN_qn_x^2)$ CNOT gates.

The implementation of $U_{12}^{(1)}(\tau)$ requires at most $2^{n_q-1} W_4W_5W_6(-\tau)$ gates, $(2^{n_q}-1)$ controlled $c-W_4W_5W_6(\tau)$ gates. Furthermore, $W_{5}$ consists of $C_{5j},j=1,\cdots,n_x$. $C_{5j}$ can be decomposed into a multi-controlled RZ gate, 2 Hadmard gates, 2 phase gates and $2j$ CNOT gates. $c-C_{5j}$ consists of $2j$ CCX gates and 4 CNOT gates. $W_4$ consists of 2 X gates, 2 controlled H gates and a CRZ gate. The implementation complexity of $W_6$ is the same as that of $W_2$. The total number of CNOT gates in the circuit implementation of $U_{12}^{(1)}(\tau)$ is determined by
$$G_{U_{12}^{(1)}} = 2^{n_q-1}(G_{W_6}+G_{W_{5}}+G_{W_4})+(2^{n_q}-1)(G_{cW_6}+G_{cW_{5}}+G_{cW_4}).$$
We have
\begin{align*}
G_{W_{5}} = &\sum_{j=1}^{n_x}2j*8+4+\sum_{j=3}^{n_x+2}(16j-40)+4n_x\\
=&16n_x^2+12n_x+4,
\end{align*}
and
\begin{align*}
  G_{cW_5} =& \sum_{j=1}^{n_x}2j*24+4*8+\sum_{j=3}^{n_x+3}(16j-40)+4n_x*8\\
=&32n_x^2+56n_x+40.
\end{align*}
Then $U_{12}^{(1)}(\tau)$ can be implemented using single qubit gates and at most $\mathcal{O}(N_qn_x^2)$ CNOT gates. The implementation of $U_{13}^{(1)}(\tau)$ requires at most $2^{n_q-1} W_3W_7(-\tau)$ gates and $(2^{n_q}-1)$ controlled $W_3W_7(\tau)$ gates. Then $U_{13}^{(1)}(\tau)$ can be implemented using single qubit gates and at most $\mathcal{O}(N_qn_x^2)$ CNOT gates. As for $U_{14}^{1}(\tau)$, we need $2^{n_q-1} cP(\tau)$ and $(2^{n_q}-1) c^2P(-\tau)$ gates. Then $U_{14}^{(1)}(\tau)$ can be implemented using single qubit gates and at most $\mathcal{O}(N_q)$ CNOT gates.

The implementations of $U_{21}^{(1)}(\tau)$ and $U_{22}^{(1)}(\tau)$ differ from those of $U_{12}^{(1)}(\tau)$ and $U_{13}^{(1)}(\tau)$ only in the phase values of the phase gates, hence the gate complexities are identical.
\end{proof}

For the case of $d=2,~ n_x=1,~ n_q=2$, we present the schematic of the quantum circuit in \autoref{Fig.0106}.
\begin{figure}[htbp] 
	\centering 
	\includegraphics[width=0.8\textwidth]{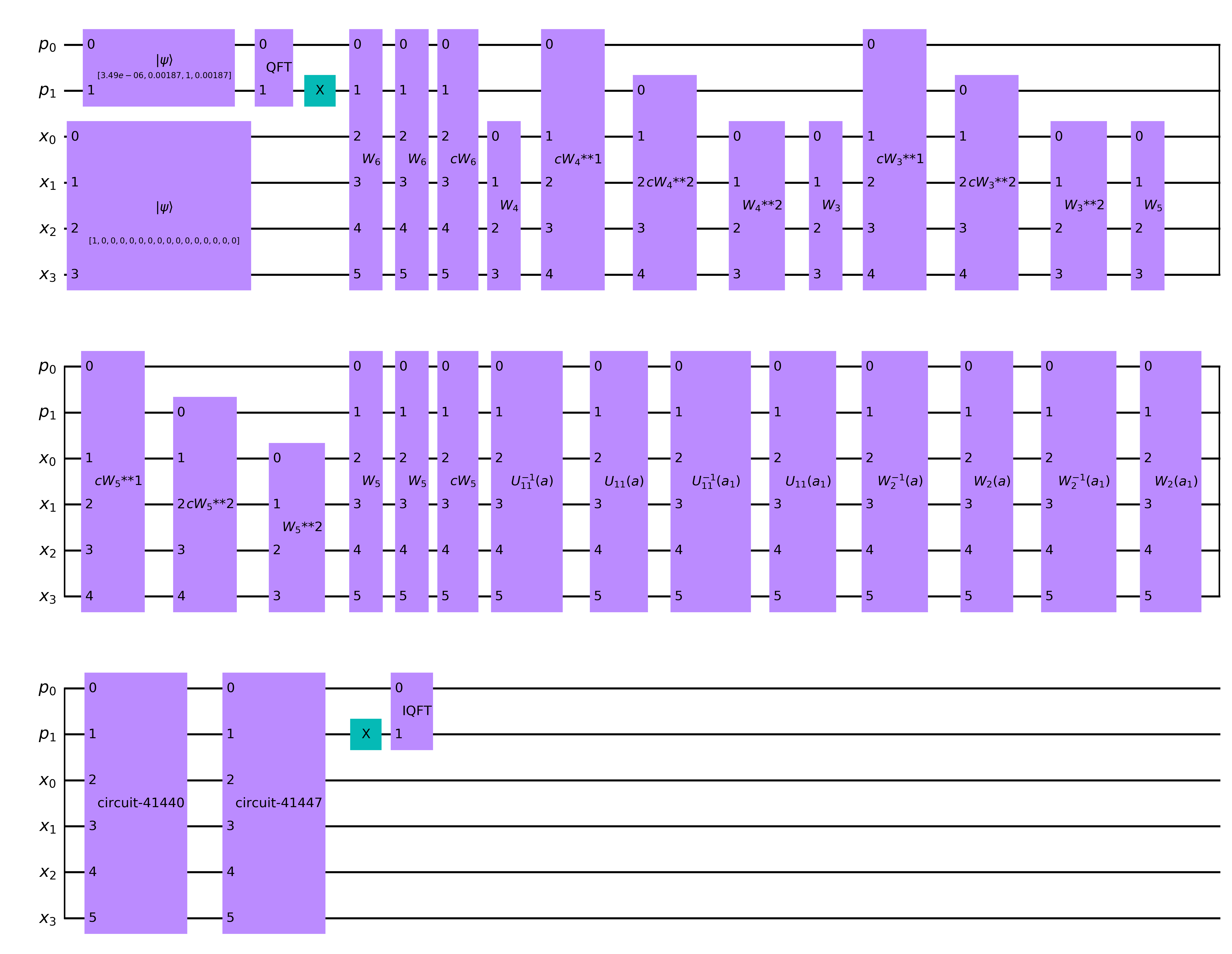} 
	\caption{Quantum circuit.} 
	\label{Fig.0106} 
\end{figure}
\subsection{Success probability estimation.}

The state $\ket{\hat{\bm{v}}_{iD}(T)},~i=1,\cdots,d$, generated by applying $V_{\text{Stokes}}(\tau)$ $r=T/\tau$ times to $\ket{\hat{\bm{v}}_i(0)}$, is transformed via an inverse quantum Fourier transform into $\ket{\bm{v}_{iD}(T)}$. A projective measurement $M_k = I^{\otimes 4n_x d} \otimes \ket{k}\bra{k}$ ($q_k > 0$) then selects the $\ket{k}$ component. From the relation $\bm{v}_i = e^{-q}\tilde{\bm{u}}_i$ ($q>0$), it follows that:
\begin{equation}
\bra{\bm{v}_{iD}(T)}M_k\ket{\bm{v}_{iD}(T)} \approx \frac{e^{-2q_k} \|\tilde{\bm{u}}_i(T)\|^2}{\|\bm{q}\|^2 \|\tilde{\bm{u}}_i(0)\|^2},
\end{equation}
yielding an approximation of $\ket{\tilde{\bm{u}}_i(T)} \ket{k}$ with probability $e^{-2q_k} \|\tilde{\bm{u}}_i(T)\|^2 / \|\bm{q}\|^2 \|\hat{\bm{u}}_i(0)\|^2$.

	\section{Complexity analysis.}
This section presents a comprehensive complexity analysis of the quantum algorithm introduced previously. We first present the upper bound for the Lie-Trotter-Suzuki decomposition error in approximating the time-evolution operator. Then we quantify the overall computational complexity accounting for discretization errors and measurement repetitions.

We first state several fundamental results: $ \sigma_{01}^{2}=\sigma_{10}^{2}=0 $, $ \sigma_{01}\sigma_{10}=\sigma_{00} $, $ \sigma_{10}\sigma_{01}=\sigma_{11}$. Direct calculation gives

\begin{equation*}
\left\|\left[\sum_{j=1}^{n_{x}}(s_{j}^{-}+s_{j}^{+}),\sigma_{01}^{\otimes n_{x}}+\sigma_{10}^{\otimes n_{x}}\right]\right\|=\left\|\left[s_{1}^{-},\sigma_{10}^{\otimes n_{x}}\right]+\left[s_{1}^{+},\sigma_{01}^{\otimes n_{x}}\right]\right\|=1.
\end{equation*}
\[
\left\|\left[\sum_{j=1}^{n_{x}}(s_{j}^{-}-s_{j}^{+}),\sigma_{01}^{\otimes n_{x}}-\sigma_{10}^{\otimes n_{x}}\right]\right\|=\left\|\left[s_{1}^{-},\sigma_{10}^{\otimes n_{x}}\right]+\left[s_{1}^{+},\sigma_{01}^{\otimes n_{x}}\right]\right\|=1.
\]

\[
\sum_{j=1}^{n_{x}}\sum_{j^{\prime}=j+1}^{n_{x}}\left\|\left[(s_{j}^{-}+s_{j}^{+}),(s_{j^{\prime}}^{-}+s_{j^{\prime}}^{+})\right]\right\|=\sum_{j^{\prime}=2}^{n_{x}}\left\|\left[s_{1}^{-},s_{j^{\prime}}^{-}\right]+\left[s_{1}^{+},s_{j^{\prime}}^{+}\right]\right\|=n_{x}-1.
\]

\[
\sum_{j=1}^{n_{x}}\sum_{j^{\prime}=j+1}^{n_{x}}\left\|\left[(s_{j}^{-}-s_{j}^{+}),(s_{j^{\prime}}^{-}-s_{j^{\prime}}^{+})\right]\right\|=\sum_{j^{\prime}=2}^{n_{x}}\left\|\left[s_{1}^{-},s_{j^{\prime}}^{-}\right]+\left[s_{1}^{+},s_{j^{\prime}}^{+}\right]\right\|=n_{x}-1.
\]

\[
\left[D_P^+, D_P^- \right]=\left[S^-+\sigma_{10}^{\otimes n_x},S^{+}+\sigma_{01}^{\otimes n_x}\right]=0
\]

\begin{lemma}\label{lem-10133}
Let $H_{\mathrm{Stokes}}$ be the Hamiltonian defined in \eqref{eqn-10161}, and consider the Schr{\"o}dinger equation
\[
\frac{d}{dt}\ket{\bm{v}(t)} = i H_{\mathrm{Stokes}} \ket{\bm{v}(t)}.
\]
Let $U_{\mathrm{Stokes}}(\tau) = \exp(i H_{\mathrm{Stokes}} \tau)$ be the exact time-evolution operator. This operator can be approximated by the Lie-Trotter-Suzuki unitary $V_{\mathrm{Stokes}}(\tau)$ defined in \eqref{eqn-10162}. The approximation error in the operator norm satisfies the bound
$$
\|U_{Stokes}(\tau) - V_{Stokes}(\tau)\| \leq \tau^2dN_q^2\gamma_1^2\gamma_2^2(n_x+1)a_1^2, 
$$
where $\gamma_1 = \frac{1}{hR},\gamma_2=\frac{1}{h},d$ denotes the spatial dimension, $N_q = 2^{n_q}$ and $N_x = 2^{n_x}$ represent the number of grid points for the variables $q$ and $x$, respectively. 
\end{lemma}

\begin{proof}
To enhance clarity, we decompose the construction of the Trotterized operator $V_{Step1}(\tau)$ into three successive approximation steps. Let
$$U_{Stokes}(\tau)\approx U_{Stokes}^{(1)}(\tau)=U_{11}^{(1)}U_{12}^{(1)}U_{13}^{(1)}U_{14}^{(1)}U_{21}^{(1)}U_{22}^{(1)}(\tau).$$
According to the theory of the Trotter splitting error with commutator scaling \cite{Com1}, we have
\begin{equation*}
\begin{aligned}
  \|U_{Stokes}(\tau)-U_{Stokes}^{(1)}(\tau)\| \leq &\frac{\tau^2}{2}(\|[H_1^{(1)}\otimes D_{\eta},(H_{1}^{(2)}+H_{1}^{(3)}) \otimes D_{\eta}+(H_{2}^{(1)}+H_{2}^{(2)})\otimes I]\|\\
&+\|[H_{1}^{(2)} \otimes D_{\eta},H_{1}^{(3)} \otimes D_{\eta}+(H_{2}^{(1)}+H_{2}^{(2)})\otimes I]\|\\
&+\|[H_{1}^{(3)} \otimes D_{\eta},H_1^{(4)}\otimes D_{\eta}+(H_{2}^{(1)}+H_{2}^{(2)})\otimes I]\|\\
&+\|[H_1^{(4)}\otimes D_{\eta},H_2^{(2)}\otimes I]\|+\|[H_2^{(1)}\otimes I,H_2^{(2)}\otimes I]\|).
\end{aligned}
\end{equation*}
From \Cref{rem-1} and the estimate of the matrix spectral norm, we have
\begin{equation*}
\begin{aligned}
  \|[H^{(1)}_1\otimes D_{\eta},H_1^{(2)}\otimes D_{\eta}]\| =&\|[\sum_{k=0}^{Nq-1}(k-\frac{N_q}{2})H^{(1)}_1\otimes \ket{k}\bra{k},\sum_{k=0}^{Nq-1}(k-\frac{N_q}{2})H^{(2)}_1\otimes \ket{k}\bra{k}]\|\\
= &\|\sum_{k=0}^{Nq-1}(k-\frac{N_q}{2})^2[H^{(1)}_1,H^{(2)}_1]\otimes \ket{k}\bra{k}\|\\
\leq & \max\limits_{0\leq k \leq N_q-1}(k-\frac{N_q}{2})^2\|[H^{(1)}_1,H^{(2)}_1]\|\\
=&\frac{N_q^2 \gamma_1^2 \gamma_2}{8 \varepsilon}\|\ket{00}\bra{01}\otimes (D_P^{+})_i H+\ket{10}\bra{00}\otimes H(D_P^{-})_i\|\\
\leq &\frac{1}{2 \varepsilon}N_q^2 \gamma_1^2 \gamma_2(d+1)(n_x+2).
\end{aligned}
\end{equation*}
\begin{equation*}
\begin{aligned}
  \|[H^{(1)}_1\otimes D_{\eta},H_1^{(3)} \otimes D_{\eta}]\| =&\|[\sum_{k=0}^{Nq-1}(k-\frac{N_q}{2})H^{(1)}_1\otimes \ket{k}\bra{k},\sum_{k=0}^{Nq-1}(k-\frac{N_q}{2})H^{(3)}_1\otimes \ket{k}\bra{k}]\|\\
\leq&\frac{N_q^2 \gamma_1^2}{8}\|a(\ket{00}\bra{10}-\ket{10}\bra{00})\otimes H+a_1(\ket{01}\bra{11}-\ket{11}\bra{01})\otimes H  \|\\
\leq&\frac{1}{2}N_q^2\gamma_1^2(a+a_1).
\end{aligned}
\end{equation*}
\begin{equation*}
\begin{aligned}
  \|[H^{(2)}_1\otimes D_{\eta},H_1^{(3)} \otimes D_{\eta}]\| =&\|[\sum_{k=0}^{Nq-1}(k-\frac{N_q}{2})H^{(2)}_1\otimes \ket{k}\bra{k},\sum_{k=0}^{Nq-1}(k-\frac{N_q}{2})H^{(3)}_1\otimes \ket{k}\bra{k}]\|\\
\leq&\frac{N_q^2 \gamma_1^2}{16\gamma_2}\|\ket{00}\bra{11} \otimes (D_P^+)_i-\ket{11}\bra{00}\otimes (D_P^-)_i\\
&+\ket{01}\bra{10}\otimes (D_P^-)_i-\ket{10}\bra{01}\otimes (D_P^+)_i  \|\\
\leq&\frac{N_q^2 \gamma_1^2}{8 \gamma_2}(n_x+2).
\end{aligned}
\end{equation*}
\begin{equation*}
\begin{aligned}
  \|[H^{(2)}_1\otimes D_{\eta},H_2^{(1)} \otimes I^{\otimes N_q}]\| =&\|[\sum_{k=0}^{Nq-1}(k-\frac{N_q}{2})H^{(2)}_1\otimes \ket{k}\bra{k},H^{(1)}_2\otimes I^{\otimes N_q}]\|\\
\leq & \max\limits_{0\leq k \leq N_q-1}|k-\frac{N_q}{2}|\|[H^{(1)}_1,H^{(2)}_2]\|\\
=&\frac{N_q \gamma_1 \gamma_2}{8}\|\left(\ket{00}\bra{00}\otimes (D_P^{+}D_P^-)_i +\ket{01}\bra{01}\otimes (D_P^{-}D_P^+)_i\right)\|\\
\leq&N_q \gamma_1 \gamma_2.
\end{aligned}
\end{equation*}
\begin{equation*}
\begin{aligned}
  \|[H^{(3)}_1\otimes D_{\eta},H_1^{(4)} \otimes D_{\eta}]\| =&\|[\sum_{k=0}^{Nq-1}(k-\frac{N_q}{2})H^{(3)}_1\otimes \ket{k}\bra{k},\sum_{k=0}^{Nq-1}(k-\frac{N_q}{2})H^{(4)}_1\otimes \ket{k}\bra{k}]\|\\
\leq&\frac{N_q^2 \gamma_1^2}{8\gamma_2^2}\|(\ket{10}\bra{00}+\ket{11}\bra{01}-\ket{00}\bra{10}-\ket{01}\bra{11}) \otimes I\|\\
\leq&\frac{N_q^2 \gamma_1^2}{8 \gamma_2^2}.
\end{aligned}
\end{equation*}
\begin{equation*}
\begin{aligned}
  \|[H^{(3)}_1\otimes D_{\eta},H_2^{(2)} \otimes I^{\otimes N_q}]\| =&\|[\sum_{k=0}^{Nq-1}(k-\frac{N_q}{2})H^{(3)}_1\otimes \ket{k}\bra{k},H^{(2)}_2\otimes I^{\otimes N_q}]\|\\
=&\frac{N_q \gamma_1 }{4 \gamma_2}\|{\rm i}\left(\ket{00}\bra{00}+\ket{01}\bra{01}-\ket{10}\bra{10}-\ket{11}\bra{11}\right)\otimes I\|\\
\leq&\frac{N_q \gamma_1 }{4 \gamma_2}.
\end{aligned}
\end{equation*}
By analogy, a similar conclusion holds
\begin{equation*}
\begin{aligned}
  \|[H^{(1)}_1\otimes D_{\eta},H_2^{(1)} \otimes I^{\otimes N_q}]\| =&\|[\sum_{k=0}^{Nq-1}(k-\frac{N_q}{2})H^{(1)}_1\otimes \ket{k}\bra{k},H^{(1)}_2\otimes I^{\otimes N_q}]\|\\
\leq&\frac{N_q \gamma_1 \gamma_2^2}{4}\|{\rm i}\left(\ket{00}\bra{01}\otimes (D_P^{+})_i H-\ket{10}\bra{00}\otimes H(D_P^{-})_i\right)\|\\
\leq&\frac{1}{\varepsilon}N_q \gamma_1 \gamma_2^2 (d+1)(n_x+2).
\end{aligned}
\end{equation*}
\begin{equation*}
\begin{aligned}
  \|[H^{(1)}_1\otimes D_{\eta},H_2^{(2)} \otimes I^{\otimes N_q}]\| =&\|[\sum_{k=0}^{Nq-1}(k-\frac{N_q}{2})H^{(1)}_1\otimes \ket{k}\bra{k},H^{(2)}_2\otimes I^{\otimes N_q}]\|\\
\leq&\frac{N_q \gamma_1 \gamma_2^2}{4}\|a(\ket{00}\bra{10}+\ket{10}\bra{00})\otimes H+a_1(\ket{01}\bra{11}+\ket{11}\bra{01})\otimes H\|\\
\leq&N_q\gamma_1 \gamma_2(a+a_1).
\end{aligned}
\end{equation*}
\begin{equation*}
\begin{aligned}
  \|[H^{(2)}_1\otimes D_{\eta},H_2^{(2)} \otimes I^{\otimes N_q}]\| =&\|[\sum_{k=0}^{Nq-1}(k-\frac{N_q}{2})H^{(2)}_1\otimes \ket{k}\bra{k},H^{(2)}_2\otimes I^{\otimes N_q}]\|\\
\leq&\frac{1}{4}N_q\gamma_1 (n_x+2).
\end{aligned}
\end{equation*}
\begin{equation*}
\begin{aligned}
  \|[H^{(3)}_1\otimes D_{\eta},H_2^{(1)} \otimes I^{\otimes N_q}]\| =&\|[\sum_{k=0}^{Nq-1}(k-\frac{N_q}{2})H^{(2)}_1\otimes \ket{k}\bra{k},H^{(2)}_2\otimes I^{\otimes N_q}]\|\\
\leq&\frac{1}{4}N_q\gamma_1 (n_x+2).
\end{aligned}
\end{equation*}
\begin{equation*}
\begin{aligned}
  \|[H^{(4)}_1\otimes D_{\eta},H_2^{(2)} \otimes I^{\otimes N_q}]\| =&\|[\sum_{k=0}^{Nq-1}(k-\frac{N_q}{2})H^{(4)}_1\otimes \ket{k}\bra{k},H^{(2)}_2\otimes I^{\otimes N_q}]\|\\
\leq&\frac{N_q \gamma_1}{4 \gamma_2}.
\end{aligned}
\end{equation*}
\begin{equation*}
\begin{aligned}
  \|[H^{(1)}_2\otimes I^{\otimes N_q},H_2^{(2)} \otimes I^{\otimes N_q}]\| =&\|[H^{(1)}_2\otimes I^{\otimes N_q},H^{(2)}_2\otimes I^{\otimes N_q}]\|\\
\leq&\frac{\gamma_1^2}{\gamma_2}(n_x+2).
\end{aligned}
\end{equation*}

We arrive at the following conclusion
\begin{equation}\label{eqn-10163}
\begin{aligned}
  \|U_{Stokes}(\tau)-U_{Stokes}^{(1)}(\tau)\| \leq &\frac{\tau^2}{2}(\frac{1}{2 \varepsilon}N_q^2 \gamma_1^2 \gamma_2(d+1)(n_x+2)+\frac{1}{2}N_q^2\gamma_1^2(a+a_1)+\frac{N_q^2 \gamma_1^2}{8 \gamma_2}(n_x+3)\\
&+\frac{1}{\varepsilon}N_q \gamma_1 \gamma_2^2 (d+1)(n_x+2)\\
&+N_q\gamma_1 \gamma_2(a+a_1+1)+\frac{1}{2}N_q\gamma_1 (n_x+3)+\frac{\gamma_1^2}{\gamma_2}(n_x+2)\\
\leq& \frac{\tau^2dN_q^2\gamma_1^2 \gamma_2^2(n_x+1)a_1^2}{2}.
\end{aligned}
\end{equation}

As the circuit design indicates, $U_{14}^{(1)}(\tau)$ can be implemented directly, while the remaining parts require further decomposition via the Lie-Trotter-Suzuki formula. We have the second operator

$$U_{Step1}^{(1)} \approx U_{Step1}^{(2)}:=U_{11}^{(2)}(a,\tau)U_{11}^{(2)}(a_1,\tau)U_{12}^{(2)}U_{13}^{(2)}U_{14}^{(1)}U_{21}^{(2)}U_{22}^{(2)}(\tau).$$
where
\begin{equation}
\begin{aligned}
U_{11}^{(2)}(a,\tau) =& \prod_{\alpha=1}^{d} \exp\{{\rm i}a \tau \gamma_2^2\ket{00}\bra{00} \otimes (S^++S^--2I^{\otimes n_x})_{\alpha} \otimes D_{\eta}\}\\
& \cdot \prod_{\alpha=1}^{d} \exp\{{\rm i}a \tau \gamma_2^2 \ket{00}\bra{00} \otimes (\sigma_{10}^{\otimes n_x}+\sigma_{01}^{\otimes n_x})_{\alpha} \otimes D_{\eta}\},\nonumber
\end{aligned}
\end{equation}
\begin{equation}
\begin{aligned}
U_{11}^{(2)}(a_1,\tau) =& \prod_{\alpha=1}^{d} \exp\{{\rm i}a_1 \tau \gamma_2^2\ket{01}\bra{01} \otimes (S^++S^--2I^{\otimes n_x})_{\alpha} \otimes D_{\eta}\}\\
& \cdot \prod_{\alpha=1}^{d} \exp\{{\rm i}a_1 \tau \gamma_2^2 \ket{01}\bra{01} \otimes (\sigma_{10}^{\otimes n_x}+\sigma_{01}^{\otimes n_x})_{\alpha} \otimes D_{\eta}\},\nonumber
\end{aligned}
\end{equation}
\begin{equation*}
\begin{aligned}
U_{12}^{(2)}(\tau) =& \exp\{{\rm i}\tau \gamma_2\left(\begin{array}{cccc}
0 & \frac{1}{2}(S^{-})_i & 0 & 0\\
\frac{1}{2}(S^{+})_i & 0 & 0 & 0 \\
0 & 0 & 0 & 0 \\
0 & 0 & 0 & 0 
\end{array}\right)\otimes D_{\eta}\}\cdot \exp\{{\rm i}\tau \gamma_2\left(\begin{array}{cccc}
0 & -\frac{1}{2}I & 0 & 0\\
-\frac{1}{2}I & 0 & 0 & 0 \\
0 & 0 & 0 & 0 \\
0 & 0 & 0 & 0 
\end{array}\right)\otimes D_{\eta}\}\\
&\cdot \exp\{{\rm i}\tau \gamma_2\left(\begin{array}{cccc}
0 & \frac{1}{2}\sigma_{10}^{\otimes n_x} & 0 & 0\\
\frac{1}{2}\sigma_{01}^{\otimes n_x} & 0 & 0 & 0 \\
0 & 0 & 0 & 0 \\
0 & 0 & 0 & 0 
\end{array}\right)\otimes D_{\eta}\}
\end{aligned}
\end{equation*}
\begin{equation*}
\begin{aligned}
U_{21}^{(2)}(\tau) =& \exp\{{\rm i}\tau \gamma_2\left(\begin{array}{cccc}
0 & -\frac{{\rm i}}{2}(S^{-})_i & 0 & 0\\
\frac{{\rm i}}{2}(S^{+})_i & 0 & 0 & 0 \\
0 & 0 & 0 & 0 \\
0 & 0 & 0 & 0 
\end{array}\right)\otimes I^{\otimes n_q}\}\cdot \exp\{{\rm i}\tau \gamma_2\left(\begin{array}{cccc}
0 & \frac{{\rm i}}{2}I & 0 & 0\\
-\frac{{\rm i}}{2}I & 0 & 0 & 0 \\
0 & 0 & 0 & 0 \\
0 & 0 & 0 & 0 
\end{array}\right)\otimes I^{\otimes n_q}\}\\
&\cdot \exp\{{\rm i}\tau \gamma_2\left(\begin{array}{cccc}
0 & -\frac{{\rm i}}{2}(\sigma_{10}^{\otimes n_x})_i & 0 & 0\\
\frac{{\rm i}}{2}(\sigma_{01}^{\otimes n_x})_i & 0 & 0 & 0 \\
0 & 0 & 0 & 0 \\
0 & 0 & 0 & 0 
\end{array}\right)\otimes I^{\otimes n_q}\},
\end{aligned}
\end{equation*}
$$U_{13}^{(2)}(\tau) = \exp\{{\rm i}\tau \left(\begin{array}{cccc}
0 & 0 & \frac{1}{2}I & 0\\
0 & 0 & 0 & 0 \\
\frac{1}{2}I & 0 & 0 & 0 \\
0 & 0 & 0 & 0 
\end{array}\right)\otimes D_{\eta}\}\cdot \exp\{{\rm i}\tau \left(\begin{array}{cccc}
0 & 0 & 0 & 0\\
0 & 0 & 0 & \frac{1}{2}I \\
0 & 0 & 0 & 0 \\
0 & \frac{1}{2}I & 0 & 0 
\end{array}\right)\otimes D_{\eta}\},$$
$$U_{22}^{(2)}(\tau) = \exp\{{\rm i}\tau \left(\begin{array}{cccc}
0 & 0 & -\frac{{\rm i}}{2}I & 0\\
0 & 0 & 0 & 0 \\
\frac{{\rm i}}{2}I & 0 & 0 & 0 \\
0 & 0 & 0 & 0 
\end{array}\right)\otimes I^{\otimes n_q}\}\cdot \exp\{{\rm i}\tau \left(\begin{array}{cccc}
0 & 0 & 0 & 0\\
0 & 0 & 0 & -\frac{{\rm i}}{2}I \\
0 & 0 & 0 & 0 \\
0 & \frac{{\rm i}}{2}I & 0 & 0 
\end{array}\right)\otimes I^{\otimes n_q}\},$$
for notational simplicity, we denote
$$U_{ij}^{(2)}(\tau)\triangleq \tilde{U}_{ij}^{1}(\tau)\tilde{U}_{ij}^{2}(\tau)\tilde{U}_{ij}^{3}(\tau).$$
Similarly, we have
\begin{equation*}
\begin{aligned}
  \|U_{11}^{(1)}(\tau)-U_{11}^{(2)}(a,\tau)U_{11}^{(2)}(a_1,\tau)\| \leq& \frac{dN_q^2\tau^2 \gamma_1^2 \gamma_2^2 (a^2+a_1^2)}{8}\|[S^++S^-,\sigma_{01}^{\otimes n_x}+\sigma_{10}^{\otimes n_x}]\|\\
=&\frac{dN_q^2\tau^2 \gamma_1^2 \gamma_2^2 (a^2+a_1^2)}{8}\|[s_1^+,\sigma_{10}^{\otimes n_x}]+[s_1^-,\sigma_{01}^{\otimes n_x}]\|\\
\leq &\frac{dN_q^2\tau^2 \gamma_1^2 \gamma_2^2 (a^2+a_1^2)}{8}. \\
  \|U_{12}^{(1)}(\tau)-U_{12}^{(2)}(\tau)\| \leq& \frac{N_q^2 \tau^2 \gamma_1^2}{32}(\|\left(\begin{array}{cccc}
S^{+}-S^{-} &0 & 0 & 0\\
0 & S^{-}-S^{+} & 0 & 0 \\
0 & 0 & 0 & 0 \\
0 & 0 & 0 & 0 
\end{array}\right)\|+1)\\
\leq &\frac{N_q^2 \tau^2 \gamma_1^2 (n_x+1)}{32},\\
\|U_{21}^{(1)}(\tau)-U_{21}^{(2)}(\tau)\| \leq & \frac{N_q \tau^2 \gamma_2^2 (n_x+1)}{16},
\end{aligned}
\end{equation*}
so
\begin{equation}\label{eqn-10164}
\|U_{Step1}^{(1)}(\tau)-U_{Step2}^{(2)}(\tau)\| \leq \frac{dN_q^2\tau^2 \gamma_1^2 \gamma_2^2 (a^2+a_1^2)}{8}+\frac{N_q^2 \tau^2 \gamma_1^2 (n_x+1)}{32}+\frac{N_q \tau^2 \gamma_2^2 (n_x+1)}{16}.
\end{equation}

In the third step, we decompose and implement the shift operator based on the formulation given in \eqref{eqn-10131}
\begin{equation}\label{eqn-10162}
U_{Stokes}^{(2)}(\tau)\approx V_{Stokes}(\tau)=V_{11}^{(1)}\tilde{U}^{(2)}_{11}(a,\tau)V_{11}^{(1)}\tilde{U}^{(2)}_{11}(a_1,\tau)V_{12}^{(1)}\tilde{U}_{12}^{(2)}\tilde{U}_{13}^{(1)}\tilde{U}_{13}^{(2)}V_{21}^{(1)}\tilde{U}_{21}^{(2)}\tilde{U}_{22}^{(1)}\tilde{U}_{22}^{(2)}(\tau).
\end{equation}
We notice that 
$$[\left(\begin{array}{cccc}
0 &(s_{j}^{-})_i & 0 & 0\\
(s_{j}^{+})_i & 0 & 0 & 0 \\
0 & 0 & 0 & 0 \\
0 & 0 & 0 & 0 
\end{array}\right),\left(\begin{array}{cccc}
0 &(s_{j^{'}}^{-})_i & 0 & 0\\
(s_{j^{'}}^{+})_i & 0 & 0 & 0 \\
0 & 0 & 0 & 0 \\
0 & 0 & 0 & 0 
\end{array}\right)]=0,j\neq j^{'},$$
then
$$V_{12}^{(1)}=\tilde{U}_{12}^{(1)},V_{21}^{(1)}=\tilde{U}_{21}^{(1)},$$
according to \cite{Cir3}, we have
\begin{equation}\label{eqn-10165}
\begin{aligned}
  \|V_{11}^{(1)}(a,\tau)-U_{11}^{(1)}(a,\tau)\| =&\| \ket{00}\bra{00} \otimes \sum_{k=1}^{N_q-1}(\prod_{\alpha=1}^{d}(\exp({\rm i} a\tau D_P^{\Delta})))_{\alpha}^{k-\frac{N_q}{2}}\otimes \ket{k}\bra{k}+(I-\ket{00}\bra{00})\otimes I\\
&- \| \ket{00}\bra{00} \otimes \sum_{k=1}^{N_q-1}(\prod_{\alpha=1}^{d}(V_0(\tau)))_{\alpha}^{k-\frac{N_q}{2}}\otimes \ket{k}\bra{k}-(I-\ket{00}\bra{00})\otimes I\|\\
\leq& \frac{dN_q \gamma_1 \gamma_2^2 \tau^2 a^2(n_x-1)}{4}.
\end{aligned}
\end{equation}
\begin{equation}\label{eqn-01061}
\|V_{11}^{(1)}(a_1,\tau)-U_{11}^{(1)}(a_1,\tau)\| \leq \frac{dN_q \gamma_1 \gamma_2^2 \tau^2 a_1^2(n_x-1)}{4}.
\end{equation}
Substituting \eqref{eqn-10163},\eqref{eqn-10164}, \eqref{eqn-10165} and \eqref{eqn-01061} into the preceding expression yields the following result
\begin{equation}
\begin{aligned}
   \|U_{Step1}(\tau)-V_{Step1}(\tau)\| \leq& \|U_{Step1}(\tau)-U_{Step1}^{(1)}(\tau)\|+\|U_{Step1}^{(1)}(\tau)-U_{Step1}^{(2)}(\tau)\|\\
&+\|U_{Step1}^{(2)}(\tau)-V_{Step1}(\tau)\|\\
\leq& \tau^2dN_q^2\gamma_1^2\gamma_2^2(n_x+1)a_1^2.
\end{aligned}
\end{equation}
\end{proof}

\begin{theorem}\label{thm-10161}
Given the Stokes equation, the state $\ket{\bm{u}(t)}$, where $\bm{u}(T)$ denotes the classically computed solution obtained via the finite difference method with a mesh size $h$, can be prepared with the precision $\varepsilon$ using the Schr{\"o}dingerization method. This preparation can be achieved using at most $\tilde{O}\left(d^{3}T^{2}\|\bm{u}(0)\|^{4}/(h^{8}\delta^{4})\right)$ single-qubit gates and CNOT gates.
\end{theorem}
\begin{proof}
The overall quantum circuit, as constructed in the preceding sections, comprises several key components: the quantum Fourier transform, the inverse quantum Fourier transform, $d$ applications of the Trotterized evolution operator $V_{Stokes}^r(\tau)$, and a projective measurement onto the subspace $M_{\geq 0} = \sum_{q_k \geq 0} M_k$. The QFT and its inverse can each be implemented with $\mathcal{O}(n_q^2)$ controlled-phase gates, corresponding to $\mathcal{O}(n_q^2)$ CNOT gates \cite{Sch2,Sch1,Qft1}.

We now assess the complexity of simulating $V_{Stokes}^r(\tau)$. Applying the gate count from Lemma~\ref{lem-10133} with the parameters $n_x = \mathcal{O}(\log(L/h))$ dictated by the spatial discretization and $\gamma_2 = 1/h$, we find the computational cost is
\[
\mathcal{O}\left( \frac{d^3 T^2 N_q^3 \log^4(L/h)}{h^8 \delta} \right).
\]
Here $\delta$ is the target precision for the operator norm error $\| U_{Stokes}(T) - V_{Stokes}^r(\tau) \| \leq \delta$, and $T = r\tau$ is the total simulation time.

The final output error arises from two sources: the discretization of the auxiliary variable $q$ and the Lie-Trotter-Suzuki error. Let $\ket{\hat{\bm{v}}_i(T)} = U_{Stokes}(T)\ket{\hat{\bm{v}}_i(0)}$ and $\ket{\hat{\bm{v}}_{iD}(T)} = V_{Stokes}^r(\tau)\ket{\hat{\bm{v}}_i(0)}$ denote the ideal and discretized states in the Fourier space, respectively. Applying the inverse QFT yields the corresponding states in the original space, $\ket{\bm{v}_i(T)}$ and $\ket{\bm{v}_{iD}(T)}$. After projection with $M_{\geq 0}$, the discretization error is quantified by
\begin{equation*}
\frac{\| M_{\geq 0} \bm{v}_i(T) - \tilde{\bm{u}}_i(T) \otimes \bm{q}_{\geq 0} \|}{\|\tilde{\bm{u}}_i(T)\| \, \|\bm{q}_{\geq 0}\|} = \mathcal{O}\left( \frac{\pi R}{N_p} + e^{-\pi R} \right),
\end{equation*}
where $\bm{q}_{\geq 0} := \sum_{q_k \geq 0} e^{-q_k} \ket{k}$. This error stems from the numerical quadrature in the $q$-domain.

The Lie-Trotter-Suzuki error is governed by the difference between the ideal and simulated states. Using the linearity of the measurement operator and the unitary invariance of the norm, we obtain
\[
\left\| M_{\geq 0} \ket{\bm{v}_{iD}(T)} - M_{\geq 0} \ket{\bm{v}_i(T)} \right\| \leq \left\| \ket{\bm{v}_{iD}(T)} - \ket{\bm{v}_i(T)} \right\| = \| V_{Stokes}^r(T) - U_{Stokes}(T) \| \cdot \|\hat{\bm{v}}_i(0)\|.
\]
Normalizing by the magnitude of the post-measurement state and recalling that $\| \bm{q}_{\geq 0} \| = \mathcal{O}(\| \bm{q} \|)$, this yields
\[
\frac{\| M_{\geq 0}( \ket{\bm{v}_{iD}(T)} - \ket{\bm{v}_i(T)} ) \|}{\| M_{\geq 0} \ket{\bm{v}_i(T)} \|} = \mathcal{O}\left( \frac{\|\tilde{\bm{u}}_i(0)\|}{\|\tilde{\bm{u}}_i(T)\|} \delta \right).
\end{equation*}
To bound the total error by $\mathcal{O}(\delta)$, we balance the two error sources. We set the quadrature parameters to $R = \mathcal{O}(\log(1/\delta))$ and $N_q = \mathcal{O}(R/\delta) = \tilde{\mathcal{O}}(1/\delta)$ to control the discretization error, and set the Trotter step precision to $\delta = \mathcal{O}( \|\tilde{\bm{u}}_i(T)\| \delta / \|\tilde{\bm{u}}_i(0)\| )$. The probability of successfully projecting onto the $M_{\geq 0}$ subspace is $\mathcal{O}( \|\tilde{\bm{u}}_i(T)\|^2 / \|\tilde{\bm{u}}_i(0)\|^2 )$. Applying the unitary $V^r_{Stokes}(T)$ independently across each of the dspatial dimensions yields the overall complexity.
\end{proof}
\begin{remark}
We consider the classical implementation of the artificial compressibility formulation for the incompressible Stokes system. After spatial discretization on the staggered grid with $2^{n_x d}$ cells, the application of the discrete differential operators requires
$
O(s\, d\, 2^{n_xd})
$
arithmetic operations per time step, where $s=O(d)$ denotes the sparsity of the differential operator. 

From the first-order temporal truncation error and the Courant-Friedrichs-Lewy (CFL) stability condition, the total number of time steps up to time $T$ within the additive error $\delta$ is
\[
O\!\left(\frac{T^2}{\delta} + \frac{T}{\varepsilon h^2}\right).
\]
Therefore, the overall computational complexity of the classical simulation is
\[
O\!\left(s\, d\, 2^{n_xd} \left(\frac{T^2}{\delta} + \frac{T}{\varepsilon h^2}\right)\right)
=
O\!\left(s\, d \left(\frac{T^2}{h^d \delta} + \frac{T}{\varepsilon h^{d+2}}\right)\right).
\]
If $h = O(\delta)$ and $\varepsilon = O(\delta^2)$, then the classical complexity scales as $O(\delta^{-(d+4)})$, quantum advantage can be achieved when $d$ is large (e.g., $d>8$ under the above scaling assumptions).
\end{remark}
\begin{remark}\cite{Cir3}
 The discretization in the momentum variable $p$ can be interpreted as a
Fourier spectral approximation of the factor $e^{-|p|}$. Since $e^{-|p|}$
is continuous but not differentiable at $p=0$, the resulting convergence
with respect to $p$ is only first order. A possible remedy is to replace
$e^{-|p|}$ by a smoother profile
\[
g(p)=
\begin{cases}
h(p), & p\in(-\infty,0],\\
e^{-p}, & p\in(0,+\infty),
\end{cases}
\]
where $h(p)$ is chosen so that $g\in C^k(\mathbb{R})$. In this case, the
discretization error in $p$ can be improved to
\[
O\left(\left(\frac{\pi R}{N_q}\right)^{k+1}+e^{-\pi R}\right),
\]
which is directly related to the preparation of the initial
state.
\end{remark}
\section{Numerical result.}
This section presents numerical simulations to validate the accuracy and assess the performance of the proposed quantum algorithm. First, we perform accuracy tests on benchmark problems to verify the correctness and effectiveness of the algorithm. Subsequently, we conduct a parameter convergence analysis to examine the algorithm's behavior and refine its practical performance.
\subsection{Accuracy tests.}
We first verify the core quantum circuit implementation of the differential operators on a simplified collocated grid via a benchmark problem. We then demonstrate the full algorithm's capability by solving the coupled Stokes system on staggered grids.

Given $p$, we first solve \eqref{eqn-10111} on a collocated grid as a benchmark case to assess the effectiveness of our quantum circuit implementation of the differential operators. \autoref{Fig.10151} depicts the numerical results: the exact solution is shown on the left, the numerical approximation in the center, and the pointwise error on the right. The results demonstrate the accuracy of our operator discretization, which verifies the correctness of the fundamental quantum circuit components.
\begin{figure}[htbp] 
	\centering 
	\includegraphics[width=1\textwidth]{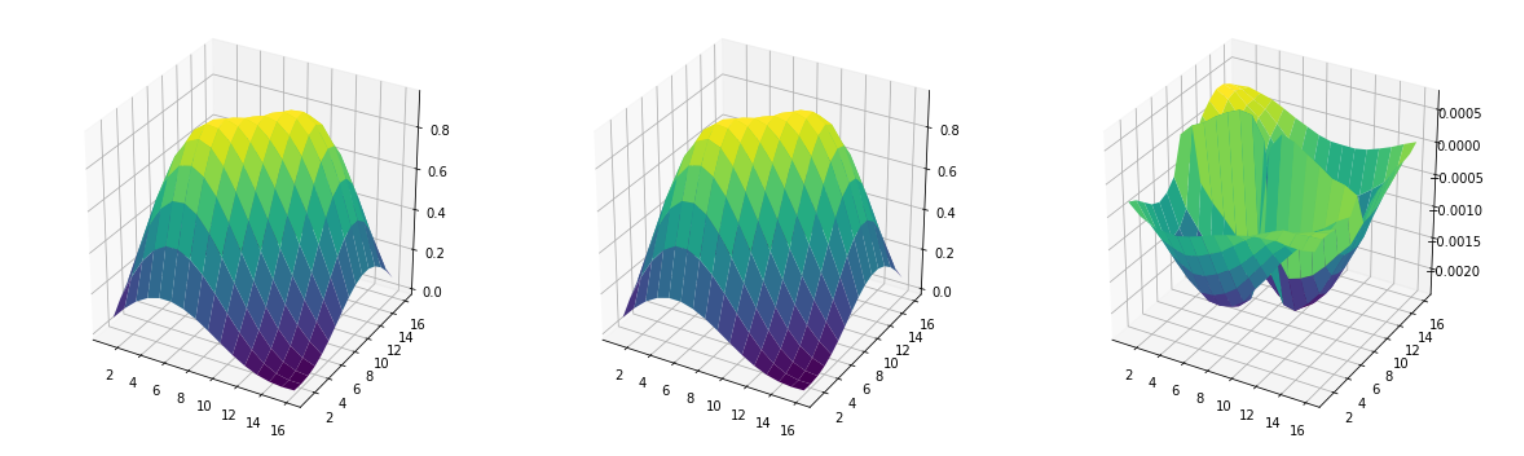} 
	\caption{Result of the collocated grid.} 
	\label{Fig.10151} 
\end{figure}
We proceed to demonstrate the full quantum algorithm for the coupled Stokes system on staggered grids, consider \eqref{eqn-01062} with the analytic solution 
$$\bm{u} = \begin{pmatrix}
e^{-t}\sin{(\frac{2 \pi x}{L})\cos{(\frac{2 \pi x}{L}})} \\
-e^{-t}\cos{(\frac{2 \pi x}{L})\sin{(\frac{2 \pi x}{L}})} 
\end{pmatrix},~ p=e^{-t}\frac{L}{2 \pi}\cos{(\frac{2 \pi x}{L})}\sin{(\frac{2 \pi x}{L})},~L=2^{n_x}.$$
where $a = 1,~ \varepsilon = 0.1$. We set $n_x = 4, n_p = 8, R = 10, T=0.4, dt=0.08$ for the Schr{\"o}dingerisation method. The numerical results are illustrated in \autoref{Fig.10135}: the exact solution is shown on the left, the numerical approximation in the center, and the pointwise error on the right.
\begin{figure}[htbp] 
	\centering 
	\includegraphics[width=0.95\textwidth]{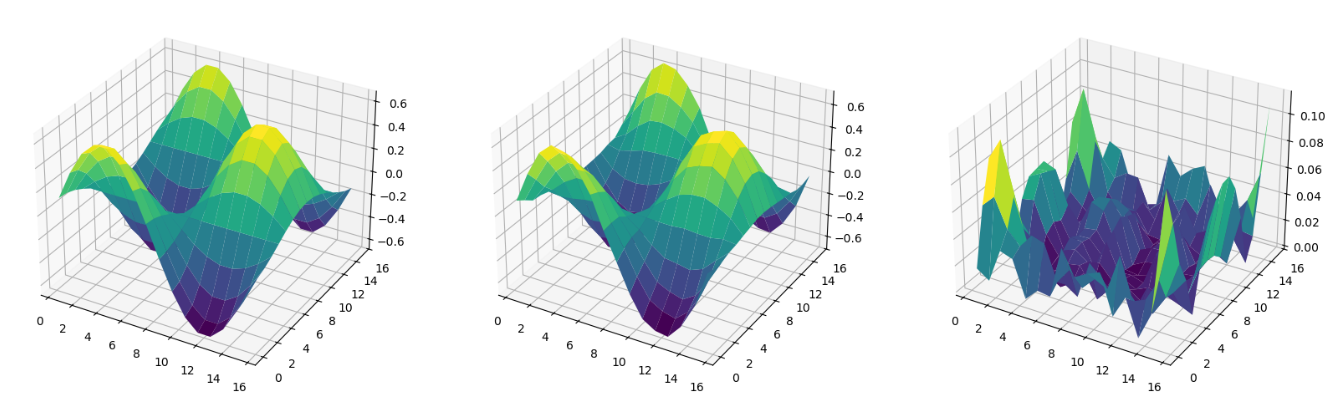} 
	\caption{Result of Stokes flow.} 
	\label{Fig.10135} 
\end{figure}
\subsection{Parameter convergence analysis.}
To further investigate the algorithm's performance and guide practical implementations, we conduct convergence tests with respect to the key parameters. Specifically, we examine the influence of the discretization parameter $n_x,~ n_q$, the time step $dt$, and the artificial compressibility parameter $\varepsilon$. For each parameter, we vary its value while keeping others fixed, and measure the error to observe the convergence trend. The results are presented as follows

We take the solution in \autoref{Fig.47} obtained with $n_x=4,~ n_q=7,~ dt=0.04,~ N_t=3,~ \varepsilon=0.2$ as the baseline solution for comparison.
\begin{figure}[htbp] 
	\centering 
	\includegraphics[width=0.95\textwidth]{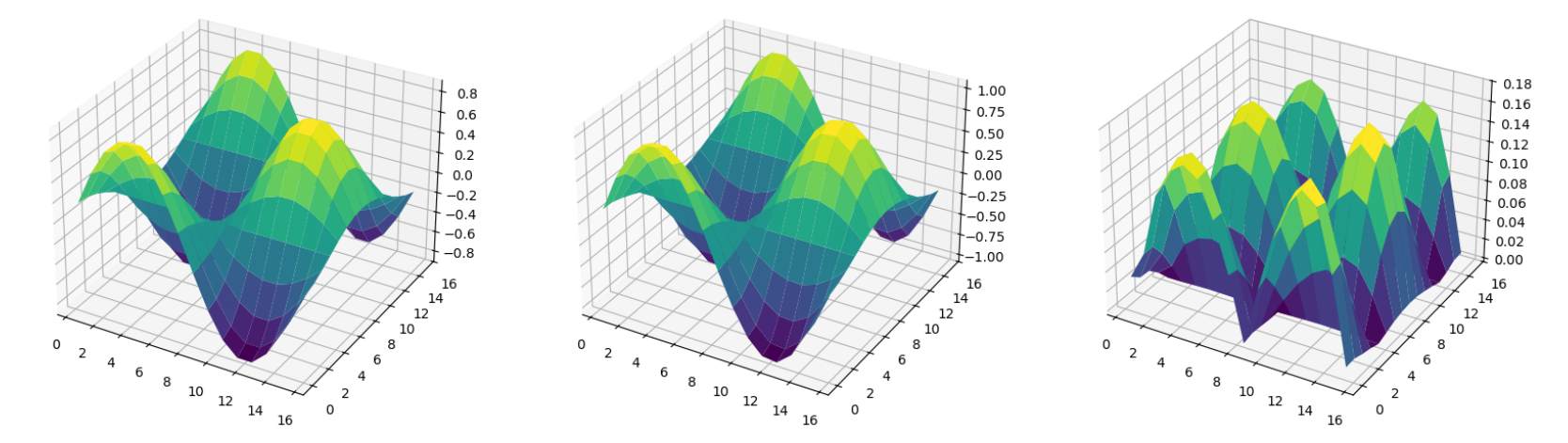} 
	\caption{$n_x=4,~ n_q=7,~ dt=0.04,~ N_t=3,~ \varepsilon=0.2.$} 
	\label{Fig.47} 
\end{figure}
First, we investigate the convergence behavior with respect to $n_q$. Keeping all other parameters fixed, we compute solutions for $n_q=6,8$. The results are shown in \autoref{Fig.46} and \autoref{Fig.48}.
\begin{figure}[htbp] 
	\centering 
	\includegraphics[width=0.95\textwidth]{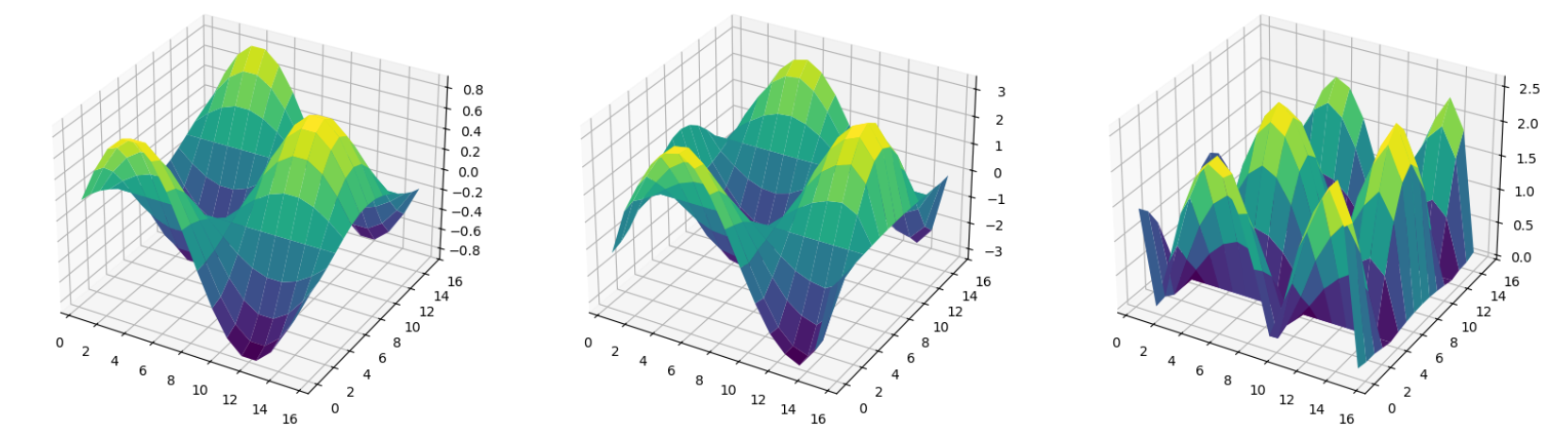} 
	\caption{$n_x=4,~ n_q=6,~ dt=0.04,~ N_t=3,~ \varepsilon=0.2.$} 
	\label{Fig.46} 
\end{figure}
\begin{figure}[htbp] 
	\centering 
	\includegraphics[width=0.95\textwidth]{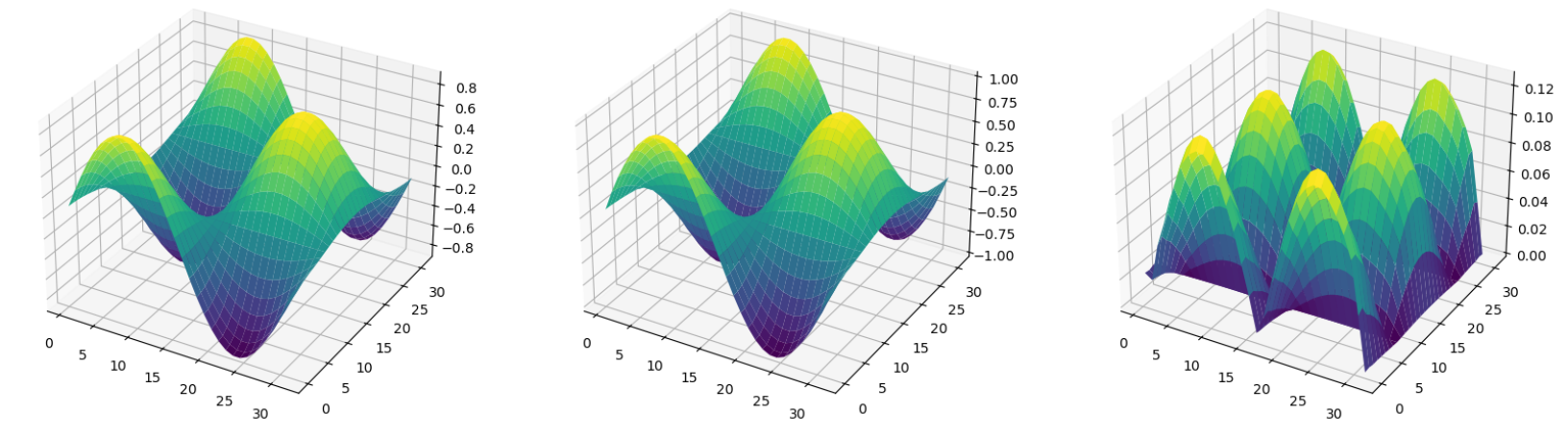} 
	\caption{$n_x=4,~ n_q=8,~ dt=0.04,~ N_t=3,~ \varepsilon=0.2.$} 
	\label{Fig.48} 
\end{figure}
Next, we turn to the spatial convergence. While keeping other parameters constant, we set
$n_x=3,5$ to observe the convergence behavior, as shown in \autoref{Fig.37} and \autoref{Fig.57}.
\begin{figure}[htbp] 
	\centering 
	\includegraphics[width=0.95\textwidth]{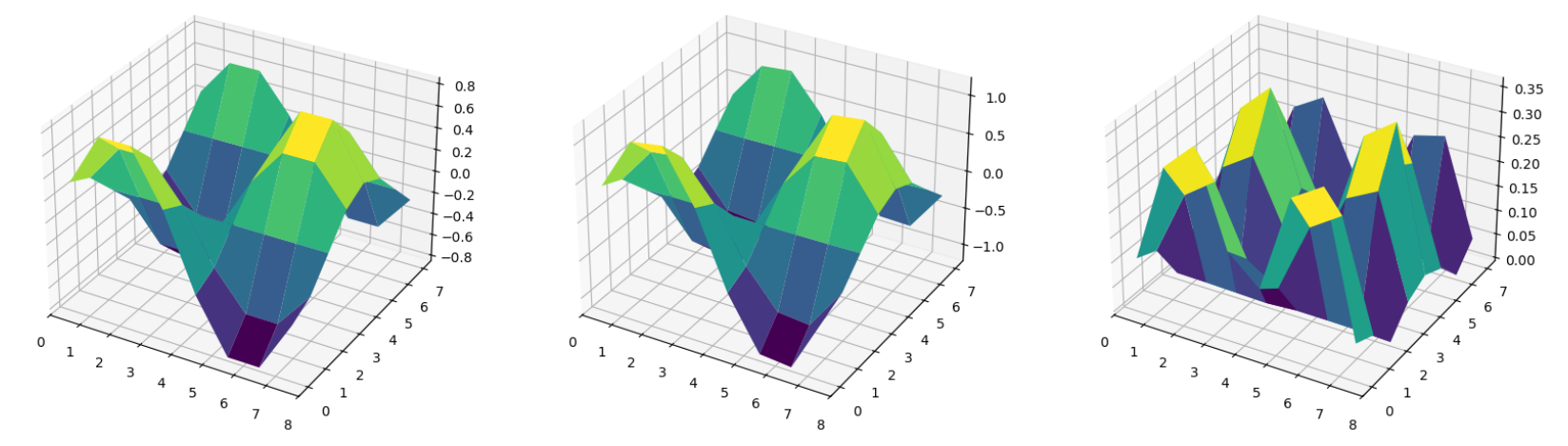} 
	\caption{$n_x=3,~ n_q=7,~ dt=0.04,~ N_t=3,~ \varepsilon=0.2.$} 
	\label{Fig.37} 
\end{figure}
\begin{figure}[htbp] 
	\centering 
	\includegraphics[width=0.95\textwidth]{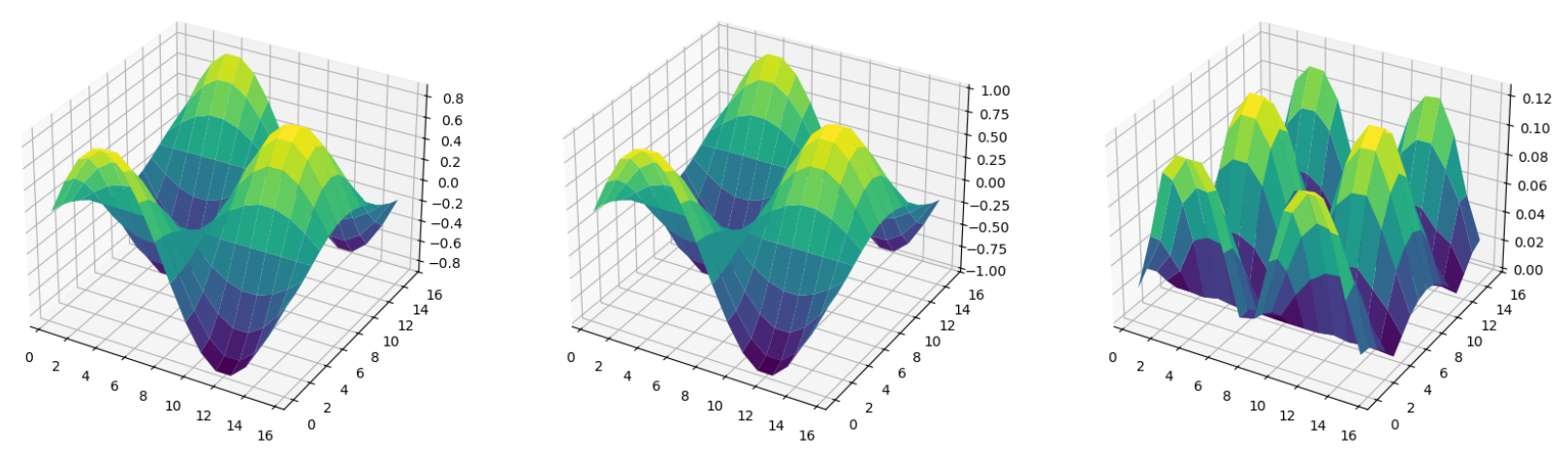} 
	\caption{$n_x=5,~ n_q=7,~ dt=0.04,~ N_t=3,~ \varepsilon=0.2.$} 
	\label{Fig.57} 
\end{figure}

The corresponding convergence behaviors for the time step $dt$ and the parameter $\varepsilon$ are presented in \autoref{Fig.dt} and \autoref{Fig.eps}.
\begin{figure}[htbp] 
	\centering 
	\includegraphics[width=0.95\textwidth]{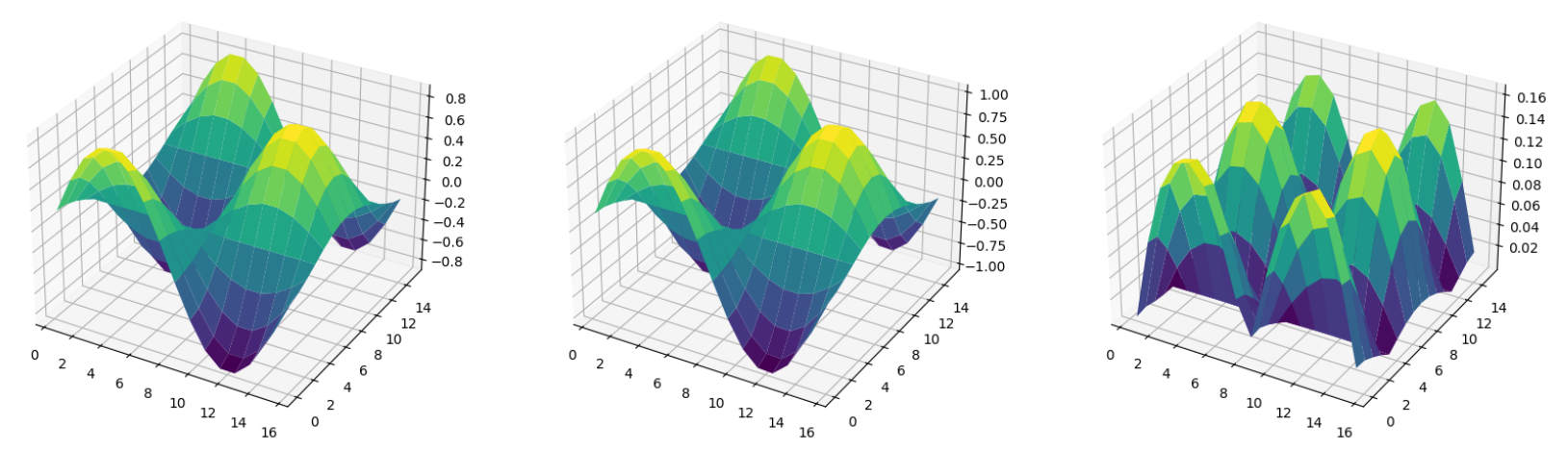} 
	\caption{$n_x=4,~ n_q=7,~ dt=0.02,~ N_t=6,~ \varepsilon=0.2.$} 
	\label{Fig.dt} 
\end{figure}
\begin{figure}[htbp] 
	\centering 
	\includegraphics[width=0.95\textwidth]{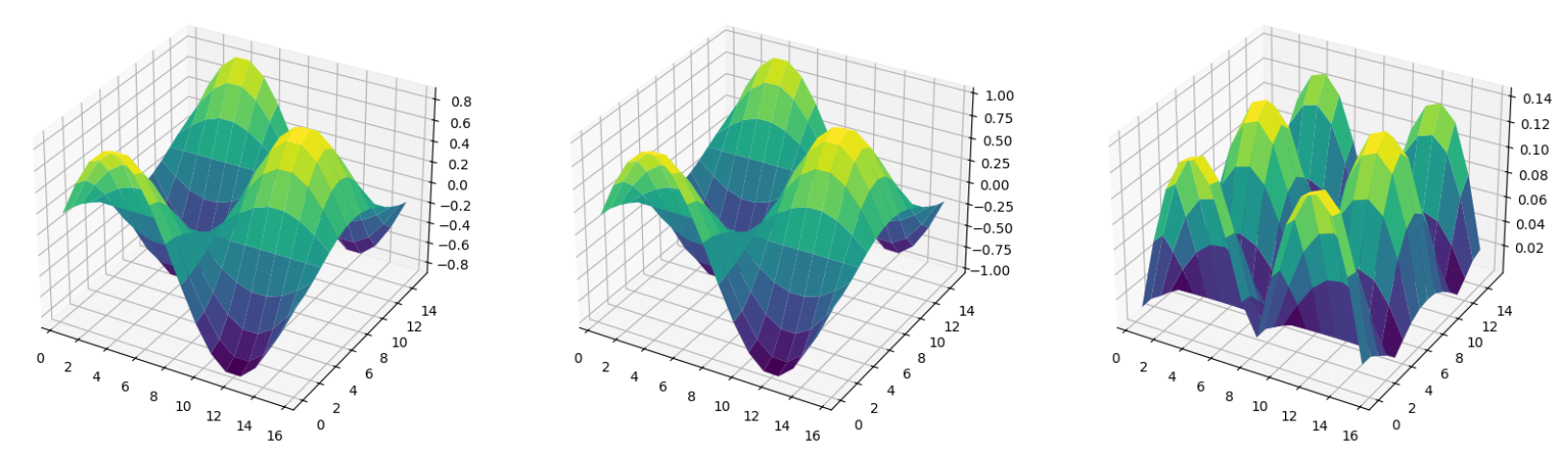} 
	\caption{$n_x=4,~ n_q=7,~ dt=0.04,~ N_t=3,~ \varepsilon=0.1.$} 
	\label{Fig.eps} 
\end{figure}
In summary, we observe that the solution error monotonically decreases as the discretization parameter $n_x,n_q$ increase, and as the time step $dt$ and the artificial compressibility parameter $\varepsilon$ decrease. The observed discrepancy originates from the finite resolution of the computational mesh. While increasing the number of grid points would systematically reduce this error, the associated demand for quantum resources exceeds the scope of our current simulation.
	\section{Conclusion.}
	This work has established a comprehensive framework for solving high-dimensional incompressible Stokes equations by integrating the Schrödingerisation technique with an artificial compressibility regularization. The core innovation lies in the explicit design of a quantum circuit that encodes the resulting non-Hermitian operator into a Schrödinger-type system, thereby circumventing the fundamental limitations of classical methods in handling the saddle-point problem. This design is supported by a staggered-grid discretization, which contributes to the numerical stability and resource efficiency of the overall approach.
A rigorous complexity analysis demonstrates a provable quantum advantage for high-dimensional problems, notably an exponential speedup in scaling with the problem dimensionality, a conclusion strongly supported by numerical simulations conducted on the Qiskit platform.

Future research will focus on porting the algorithm to near-term quantum hardware to assess its practical resilience and on generalizing the framework to tackle the nonlinearities inherent in the full Navier-Stokes equations.

\bibliographystyle{siam} 
\bibliography{reference1}

\newpage
\appendix
\section{Detailed Derivations for the Artificial Compressibility Method.}
\label{app1}
\subsection{Error Analysis for the Steady Stokes Problem} 
\label{app:proof_steady}
Let $\bm{e}:=\bm{u}^{\varepsilon}-\bm{u},~ \rho := p^{\varepsilon}-p$. Subtracting the two systems we can get the error equation
\begin{align}
     - a\Delta\bm{e} - \nabla \rho &= \bm{0}, &&\text{in } \Omega ,\nonumber \\
    \nabla \cdot \bm{e} &= \varepsilon p^{\varepsilon} = \varepsilon(p+\rho), &&\text{in } \Omega ,\nonumber \\
    \bm{e} &= 0, &&\text{on } \partial \Omega, \nonumber
\end{align}
The corresponding weak form
\begin{align}
  a(\nabla \bm{e},\nabla \bm{v})+(\rho,\nabla \cdot \bm{v}) &= 0, ~ &&\forall \bm{v} \in [H_0^1(\Omega)]^d, \label{eqn-04141}\\
(\nabla \cdot \bm{e},q) &= \varepsilon(p+\rho,q). ~ &&\forall q \in L^2_0(\Omega) .\label{eqn-04142}
\end{align}

For \eqref{eqn-04141}, we use the inf-sup condition: there exists $\beta>0$ such that
$$\beta \|\rho\|_{L^2} \leq \sup_{\bm{v} \in [H_0^1(\Omega)]^d \backslash 0} \frac{|(q,\nabla \cdot \bm{v})|}{\|\nabla \bm{v}\|}\leq a\|\nabla \bm{e}\|_{L^2}.$$
Testing \eqref{eqn-04141} with $\bm{v}=\bm{e}$ and subtracting \eqref{eqn-04142}, we get
$$a\|\nabla \bm{e}\|^2_{L^2} \leq \varepsilon \|\rho\|_{L^2}\|p\|_{L^2}+\varepsilon\|\rho\|^2_{L^2},$$
then
$$a\|\nabla \bm{e}\|^2_{L^2} \leq \frac{\varepsilon a}{\beta} \|\nabla \bm{e}\|_{L^2}\|p\|_{L^2}+\frac{\varepsilon a^2}{\beta^2}\|\nabla \bm{e}\|^2_{L^2}.$$
Whenever $ \varepsilon < \frac{\beta^2}{a}$, we get
$$\|\nabla \bm{e}\|_{L^2} \leq \frac{\varepsilon / \beta}{1-\varepsilon a / \beta^2}\|p\|_{L^2}.$$
Then the pressure error
$$\|\rho\|_{L^2} \leq \frac{\varepsilon a/ \beta^2}{1-\varepsilon a/\beta^2}\|p\|_{L^2},$$
and from the Poincar{\'e} inequality, we have the velocity error
$$\|\bm{e}\|_{L^2} \leq C\frac{\varepsilon / \beta}{1-\varepsilon a/\beta^2}\|p\|_{L^2}.$$
\begin{theorem}
  Let $(\bm{u},p)$ solve the incompressible Stokes problem and let $(\bm{u}^{\varepsilon},p^{\varepsilon}) \in [H^1_0(\Omega)]^d \times L^2_0(\Omega)$ solve the artifical compressibility problem above. Assume the Stokes inf-sup condition with constant $\beta > 0$. Then for $\forall~ 0<\varepsilon<\frac{\beta^2}{a}$, we have
$$\bm{u}^{\varepsilon} \to \bm{u},~ p^{\varepsilon} \to p ,$$
with linear rate $O(\varepsilon)$.
\end{theorem}
\begin{example}
We conduct a 2D numerical test to justify the complete rigorous estimate of steady Stokes case
\begin{align}
     - a\Delta\bm{u} - \nabla p &= \bm{f}, &&\text{in } \Omega ,\nonumber \\
    \nabla \cdot \bm{u} &= 0, &&\text{in } \Omega , \nonumber \\
    \bm{u} &= 0, &&\text{on } \partial \Omega, \nonumber
\end{align}
where $\Omega = [0,1] \times [0,1]$, the right-hand side function $\bm{f}$ is computed to match the exact solution
$$\bm{u} = \left(\begin{array}{c}
\sin(2 \pi x)\cos(2 \pi y) \\
-\cos(2 \pi x)\sin(2 \pi y) 
\end{array}\right),~ p=-2 \pi \cos(2 \pi x)\cos(2 \pi y).$$
\end{example}
We employ the weak Galerkin finite element method \cite{ST4} combined with artificial compressibility. The scheme utilizes stably-paired, discontinuous finite element spaces: the velocity is approximated by piecewise polynomials of degree $k\geq 1$, while the pressure is approximated by piecewise polynomials of degree $k-1$. To ensure stability, the velocity space is enriched on element interfaces with polynomials of degree $k-1$.

\Cref{tab1} shows the errors and convergence rates with respect to different $\varepsilon$, when the mesh size $h=1/128$ and $k=1$. It is obvious that the convergence rates for the velocity function in $L^2$ norm and the pressure function in $L^2$ norm are of order $O(\varepsilon)$, which coincides with the theoretical analysis.
	\begin{table}[htbp]
        \footnotesize
        \caption{Errors and convergence rates.}
        \label{tab1}
		\begin{center}
			\begin{tabular}{||c|cc|cc||} 
				\hline
				$\varepsilon$ & $\|\bm{e}\|_{L^2}$ & order & $\|\rho\|_{L^2}$ & order \\
				\hline
				\hline
				1/20 & 1.45e-02 &  & 5.38e-01 &   \\
				\hline
				1/40 & 7.90e-03 & 0.87 & 2.95e-01 & 0.87  \\
				\hline
				1/80 & 4.20e-03 & 0.91 & 1.57e-01 & 0.91  \\
				\hline
				1/160 & 2.20e-03 & 0.92 & 8.18e-02 & 0.94  \\
				\hline
			\end{tabular}
		\end{center}
	\end{table}
\subsection{Error Analysis for the Time-Dependent Stokes Problem} \label{app:proof_time_dependent}
Now we turn to the incompressible time-dependent Stokes system. Assume $p \in L^2(0,T;L^2(\Omega))$. Similarly, we can get the error equation
\begin{align}
    \bm{e}_{t} - a\Delta\bm{e} - \nabla \rho &= \bm{0}, &&\text{in } \Omega \times (0, T], \label{eqn-04143}\\
    \nabla \cdot \bm{e} &= \varepsilon(p+\rho), &&\text{in } \Omega \times (0, T], \label{eqn-04144}\\
    \bm{e}(\cdot, 0) &= 0, &&\text{in } \Omega,\nonumber
\end{align}

We test \eqref{eqn-04143} with $\bm{e}$. Subtracting $\bm{e}|_{\partial \Omega}=0$ and \eqref{eqn-04144}, we have
$$(\bm{e}_t,\bm{e})-a(\Delta \bm{e},\bm{e})-\frac{1}{\varepsilon}\left(\nabla(\nabla \cdot \bm{e}),\bm{e}\right) = (\nabla p,\bm{e}),$$
then
$$
\frac{1}{2} \frac{d}{dt} \|\bm{e}\|_{L^2}^2 + a \|\nabla \bm{e}\|_{L^2}^2 + \frac{1}{\varepsilon} \|\nabla \cdot \bm{e}\|_{L^2}^2 = -(p, \nabla \cdot \bm{e}).
$$
From the Young's inequality
$$
|(p, \nabla \cdot \bm{e})| \leq \frac{1}{2\varepsilon} \|\nabla \cdot \bm{e}\|_{L^2}^2 + \frac{\varepsilon}{2} \|p\|_{L^2}^2,
$$
hence
$$
\frac{1}{2} \frac{d}{dt} \|\bm{e}\|_{L^2}^2 + a \|\nabla \bm{e}\|_{L^2}^2 + \frac{1}{2\varepsilon} \|\nabla \cdot \bm{e}\|_{L^2}^2 \leq \frac{\varepsilon}{2} \|p\|_{L^2}^2.
$$
Integrating in time from $0$ to $t$, and using $\bm{e}(\cdot,0)=0$,
$$
\|\bm{e}(t)\|_{L^2}^2 + 2a \int_0^t \|\nabla \bm{e}(s)\|_{L^2}^2 \, {\rm d}s
+ \frac{1}{\varepsilon} \int_0^t \|\nabla \cdot \bm{e}(s)\|_{L^2}^2 \, {\rm d}s
\leq \varepsilon \int_0^t \|p(s)\|_{L^2}^2 \, {\rm d}s.
$$
So the divergence defect is $O(\varepsilon)$, while the velocity error from this first energy estimate is $O(\sqrt{\varepsilon})$.

To get sharper estimations, now assume
$$ p \in L^{\infty}(0,T;L^2(\Omega)),~ p_t \in L^2(0,t;L^2(\Omega)).$$
We test \eqref{eqn-04143} with $\bm{e}_t$,
$$(\bm{e}_t,\bm{e}_t)-a(\Delta \bm{e},\bm{e}_t)+\frac{1}{\varepsilon}(\nabla(\nabla \cdot \bm{e}),\bm{e}_t) = (\nabla p,\bm{e}_t),$$
then
$$\|\bm{e}_t\|^2_{L^2}+\frac{a}{2}\frac{d}{dt}\|\nabla \bm{e}\|^2_{L^2}+\frac{1}{2 \varepsilon}\frac{d}{dt}\|\nabla \cdot \bm{e}\|^2_{L^2} = -(p,\nabla \cdot \bm{e}_t).$$
Integrate from 0 to $t$,
$$\int_0^t\|\bm{e}_t\|^2_{L^2}\, {\rm d}s+\frac{a}{2}\|\nabla \bm{e}(t)\|^2_{L^2}+\frac{1}{2\varepsilon}\|\nabla \cdot \bm{e}(t)\|^2_{L^2}=-(p(t),\nabla \cdot \bm{e}(t))+\int_0^t(p_t,\nabla \cdot \bm{e})\, {\rm d}s.$$
From the Young's inequality, we have
$$\int_0^t\|\bm{e}_t(s)\|^2_{L^2}\, {\rm d}s+a\|\nabla \bm{e}(t)\|^2_{L^2}+\frac{1}{\varepsilon}\|\nabla \cdot \bm{e}(t)\|^2_{L^2} \leq C \varepsilon \left(\|p\|^2_{L^{\infty}(0,T;L^2)}+\|p\|^2_{L^2(0,T;L^2)}+\|p_t\|^2_{L^2(0,T;L^2)}\right).$$

The weak form of \eqref{eqn-04143} is
$$
(\bm{e}_t, \bm{v}) + a (\nabla \bm{e}, \nabla \bm{v}) + (\rho, \nabla \cdot \bm{v}) = 0 
,~  \forall \bm{v} \in [H_0^1(\Omega)]^d.
$$
By the Stokes inf-sup condition,
\[
\beta \|\rho\|_{L^2} \leq \|\bm{e}_t\|_{H^{-1}} + 
a \|
\nabla \bm{e}\|_{L^2}.
\]
Integrating in time and using \(\|\bm{e}_t\|_{H^{-1}} \leq C \|\bm{e}_t\|_{L^2}\),
\[
\|\rho\|_{L^2(0,T;L^2)} \leq C \left( \|\bm{e}_t\|_{L^2(0,T;L^2)} + \|
\nabla \bm{e}\|_{L^2(0,T;L^2)} \right).
\]
Then
\[
\|p^\varepsilon - p\|_{L^2(0,T;L^2)} \leq C \sqrt{\varepsilon} \left( \|p\|_{L^\infty(0,T;L^2)} + \|p\|_{L^2(0,T;L^2)} + \|p_t\|_{L^2(0,T;L^2)} \right)
\]
\begin{theorem}
Let $(\bm{u}, p)$ solve the incompressible time-dependent Stokes system, and let $(\bm{u}^\varepsilon, p^\varepsilon)$ solve the artificial compressibility system with the same initial velocity $\bm{u}^\varepsilon(0) = \bm{u}(0)$.

If
\[
p \in L^2(0,T; L^2(\Omega)),
\]
then
\[
\|\bm{u}^\varepsilon - \bm{u}\|_{L^\infty(0,T;L^2)}^2 + \|\bm{u}^\varepsilon - \bm{u}\|_{L^2(0,T;H_0^1)}^2 + \frac{1}{\varepsilon} \|\nabla \cdot \bm{u}^\varepsilon\|_{L^2(0,T;L^2)}^2 \leq C \varepsilon \|p\|_{L^2(0,T;L^2)}^2.
\]

If in addition
\[
p \in L^\infty(0,T; L^2), \qquad p_t \in L^2(0,T; L^2),
\]
then
\[
\sup_{0 \leq t \leq T} \|
\nabla (\bm{u}^\varepsilon - \bm{u})(t)\|_{L^2}^2 + \int_0^T \|(\bm{u}^\varepsilon - \bm{u})_t\|_{L^2}^2 \, {\rm d}t + \sup_{0 \leq t \leq T} \frac{1}{\varepsilon} \|\nabla \cdot \bm{u}^\varepsilon(t)\|_{L^2}^2 \leq C \varepsilon,
\]
and
\[
\|p^\varepsilon - p\|_{L^2(0,T;L^2)} \leq C \sqrt{\varepsilon}.
\]
\end{theorem}
\section{Staggered Grid Layout} 
\label{app:grid_fig}
Using the two-dimensional case as an example, momentum equations are discretized in the horizontal and vertical directions as follows 
\begin{figure}[htbp] 
	\centering 
	\includegraphics[width=0.55\textwidth]{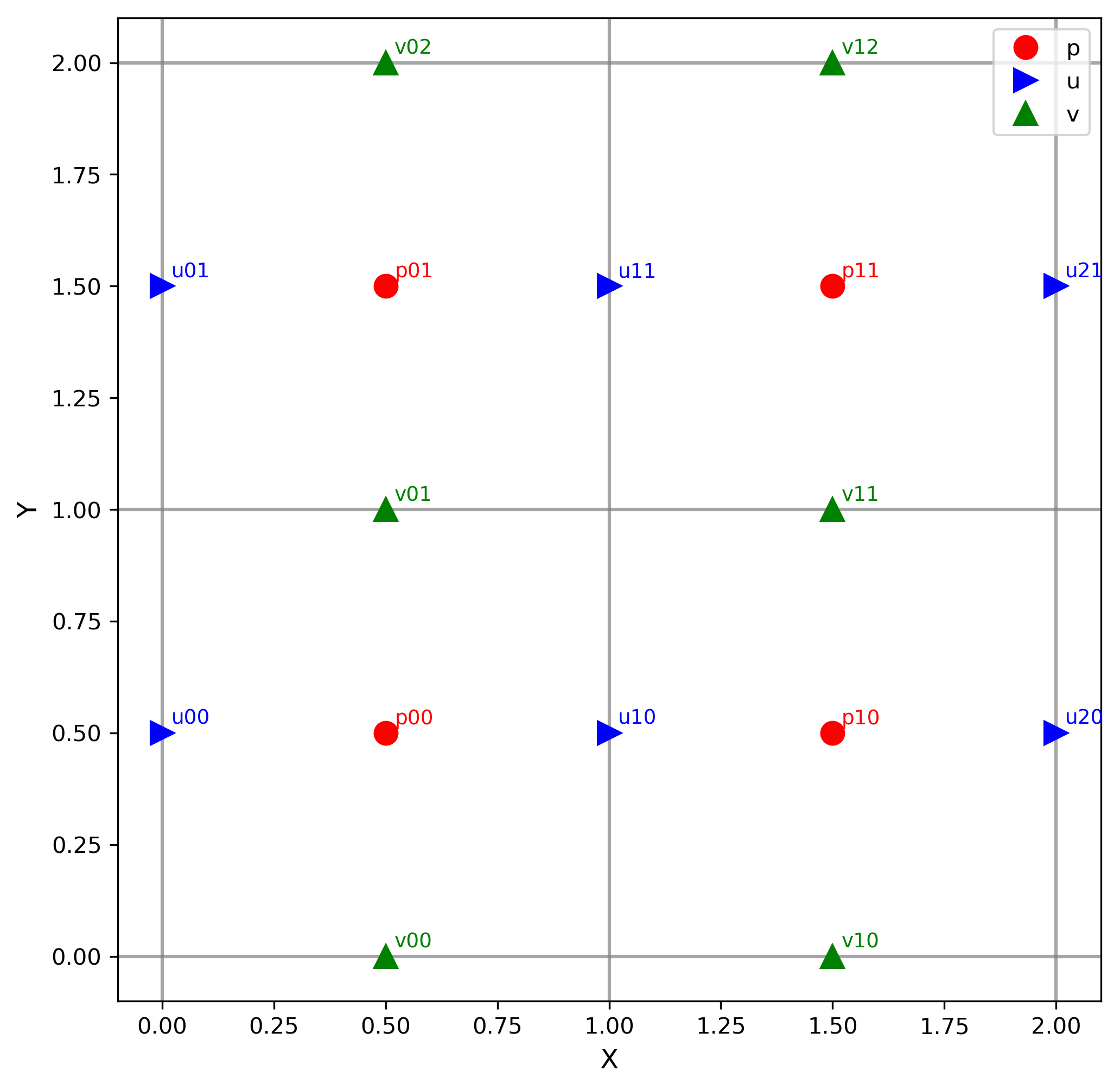} 
	\caption{Staggered grid.} 
	\label{Fig.sta} 
\end{figure}
\begin{itemize}
    \item Pressure $p$ is stored at integer grid points $(i, j)$;
    \item The $x$-direction velocity $u$ is stored at half-nodes $(i+\frac{1}{2},j)$;
    \item The $y$-direction velocity $v$ is stored at half-nodes $(i,j+\frac{1}{2})$.
\end{itemize}

\end{document}